\documentclass[11pt, reqno]{amsart}
\usepackage{amsfonts,amssymb}
\usepackage{amsmath,amscd}
 \usepackage[all]{xy}
 \usepackage{color}   
\usepackage{hyperref}
\usepackage{graphicx}
\usepackage{multicol}
\usepackage{tikz}
\usetikzlibrary{decorations.markings, arrows.meta}
\usetikzlibrary{arrows, decorations.markings}
\usetikzlibrary{arrows}
\usetikzlibrary{shapes,snakes}
\usetikzlibrary{arrows.meta}
\usetikzlibrary{decorations.pathreplacing,angles,quotes}
\tikzset{bullet/.style={draw,ellipse, text width = 4cm, text centered}}
\tikzset{rec/.style={draw, text width = 3.5 cm, text centered}}
\tikzset{plain/.style={->,>=stealth}}
\tikzset{
  barbarrow/.style={ 
     >={Straight Barb[left,length=5pt, width=5pt]}
  },
  strike through/.style={
    postaction=decorate,
    decoration={
      markings,
      mark=at position 0.5 with {
        \draw[-] (-3pt,-3pt)  --  (3pt, 3pt);
      }
    }
  }
}

\newlength{\bibitemsep}
\newlength{\bibparskip}
\let\oldthebibliography\thebibliography
\renewcommand\thebibliography[1]{%
  \oldthebibliography{#1}%
  \setlength{\parskip}{\bibitemsep}%
  \setlength{\itemsep}{\bibparskip}%
}
\newcommand{\bb}{\mathbb}

\newcommand{\cal}{\mathcal}

\newcommand{\bfk}{\mathbf{k}}
\newcommand{\cc}{\ensuremath{\mathbb{C}}}
\newcommand{\on}[1]{\operatorname{#1}}
\newcommand{\delete}[1]{{}}
\newcommand{\Mod}{\on{\!-Mod}}

\newcommand{\myhookrightarrow}{\raisebox{2pt}{ \, \begin{tikzpicture}
  \draw [right hook-latex] (0, 0) -- (0.7, 0);     
\end{tikzpicture} \, }}

\theoremstyle{plain}
\newtheorem{theorem}{Theorem}[section]
\newtheorem{lemma}[theorem]{Lemma}
\newtheorem{proposition}[theorem]{Proposition}
\newtheorem{corollary}[theorem]{Corollary}

\theoremstyle{definition}
\newtheorem{remark}[theorem]{Remark}
\newtheorem{definition}[theorem]{Definition}
\newtheorem{remarks}[theorem]{Remarks}
\newtheorem{example}[theorem]{Example}
\newtheorem{examples}[theorem]{Examples}

\delete{
\theoremstyle{plain}
\newtheorem{theorem}[subsection]{Theorem}
\newtheorem{lemma}[subsection]{Lemma}
\newtheorem{proposition}[subsection]{Proposition}
\newtheorem{corollary}[subsection]{Corollary}

\theoremstyle{definition}
\newtheorem{remark}[subsection]{Remark}
\newtheorem{definition}[subsection]{Definition}
\newtheorem{remarks}[subsection]{Remarks}
}

\newtheoremstyle{dotless}{}{}{\itshape}{}{\bfseries}{}{ }{}
\theoremstyle{dotless}
\newtheorem*{theorem*}{Theorem.}
\newtheorem*{lemma*}{Lemma}
\newtheorem*{proposition*}{Proposition}
\newtheorem*{definition*}{Definition}

\title[Cohomological varieties associated to vertex operator algebras]{Cohomological varieties associated to vertex operator algebras}
\author[A. CARADOT, C. JIANG]{Antoine CARADOT$^1$, Cuipo JIANG$^1$}
\address{$^1$School of Mathematical Sciences, Shanghai Jiao Tong University, Shanghai, 200240, China}
\author[Z. LIN]{Zongzhu LIN$^2$}
\address{$^2$Department of Mathematics, Kansas State University, Manhattan, KS, 66506, USA}
\date{\today}                     
\thanks{
 C. Jiang is supported by NSFC grant No. 12171312}

\begin{document}
\maketitle
\vspace{-\topsep}
\begin{abstract}
Given a vertex operator algebra $V$, one can attach a graded Poisson algebra called the $C_2$-algebra $R(V)$. The associate Poisson scheme provides an important invariant for $V$ and has been studied by Arakawa as the associated variety.  In this article, we define and examine the cohomological variety of a vertex algebra, a notion cohomologically dual to that of the associated variety, which measures the smoothness of the associated scheme at the vertex point.  We study its basic properties and then construct a closed subvariety of the cohomological variety for rational affine vertex operator algebras  constructed from  finite dimensional simple Lie algebras. We also determine the cohomological varieties of the simple Virasoro vertex operator algebras. These examples indicate that, although the associated variety for a rational $C_2$-cofinite vertex operator algebra is always a simple point, the cohomological variety can have as large a dimension as possible. In this paper, we study $R(V)$ as a commutative algebra only and do not use the property of its Poisson structure, which is expected to provide more refined invariants. The goal of this work is to study the cohomological supports of modules for vertex algebras as the cohomological support varieties for finite groups and restricted Lie algebras. 
\end{abstract}


\section{Introduction}
\delete{ The theory of vertex algebras originated from string theory, more particularly within conformal field theory, in an attempt to characterise particle interactions. Its mathematical formalism was developed alongside the classification of finite simple groups. This vast project divided the finite simple groups in several infinite families, as well as 26 special groups called sporadic and not belonging to any of said families. The biggest of these groups, called the Monster, was conjectured to possess very particular properties. In order to understand the Moonshine conjectures, a module for the Monster group called the moonshine module $V^\sharp$ was constructed (cf. \cite{Frenkel-Lepowsky-Meurman}). It was then noticed that there exists an algebra of "vertex operators" acting on $V^\sharp$ whose automorphism group is the Monster group itself. The space $V^\sharp$ together with this algebra of operators is a prime example of a vertex operator algebra. We refer the reader to the introduction of \cite{Lepowsky-Li} for details and references regarding the emergence of the theory of vertex algebras.

The theory of vertex (operator) algebras has then developed at a fast pace. The notion of a module for a vertex algebra was defined, and the subsequent representation theory was further explored. The algebraic structures that emerged created new deep results as well as new insight into the representation theory of known structures, like the Virasoro algebra or the affine Kac-Moody algebras (see \cite{Frenkel-Zhu, Feigin-Fuchs, Wang}).
}

Given a vertex operator algebra $V$, Zhu (\cite{Zhu}) constructed an associated algebra $A(V)$ and a Poisson algebra $R(V)=V/C_2(V)$. The irreducible representations of $A(V)$ parametrize the irreducible representations of $V$. The Poisson algebra $R(V)$ attaches an algebraic Poisson variety which is called the associated variety $X_V$ (see \cite{Arakawa1, Arakawa2}). When $V$ is  of CFT type  and strongly finitely generated, $R(V)$ is a finitely generated algebra and thus  the associated algebraic variety is always a conical Poisson variety. It is also known that $V$ is $C_2$-cofinite if and only if $ X_V$ is a single point.  In this paper, we attach another algebraic geometric object called cohomological variety using the cohomology of the algebra $R(V)$ at the vertex point.  In case $V$ is  $C_2$-cofinite of CFT type, despite the associated variety $X_V$ always being  a single point,  the cohomological variety can be quite different with large dimensions. Thus it should be an important invariant of the vertex operator algebra $V$.   We will construct a closed subvariety  of the cohomological varieties for all rational affine vertex operator algebras. This closed subvariety is an affine space of dimension being a polynomial of the positive integer level $k$. For a simple Virasoro vertex operator algebra $L_{Vir}(c,0)$, the cohomological variety is  a single point $\{0\}$ (for which the associated variety is $\mathbb A^1$)  except  when $ c=c_{p, q}$, for which, the cohomological variety is $ \mathbb A^1$ while the associated variety is $\{0\}$.  We will treat the lattice vertex operator algebra cases in a later paper. In that case, we expect $X_V$ and $X^!_V$ to be closely related to Koszul duality of quadratic algebras.

A vertex algebra, by definition, is an infinite dimensional vector space with an additional structure given by intricate axioms (cf. \cite{Borcherds, Frenkel-Lepowsky-Meurman, Lepowsky-Li}). The study of the representation theory of a vertex algebra often relies on related objects that are more familiar, with more tools available. In his thesis \cite{Zhu}, Zhu studied the representation theory of vertex operator algebras by associating to them several associative algebras, which have played prominent roles ever since.   The Zhu algebra $A(V)$ of a vertex algebra $V$, which is defined as a quotient of $V$ by a particular subspace, is an associative algebra that captures the irreducible representations as he has shown that there is a bijection between the set of equivalence classes of irreducible admissible $V$-modules and the set of equivalence classes of irreducible $A(V)$-modules.  When $V$ is rational, meaning that the category of admissible $V$-modules is semisimple, then $A(V)$ is finite dimensional and semisimple (cf. \cite{Zhu, Dong-Li-Mason1}).  This result practically establishes an equivalence of abelian categories between the category of admissible representations of $V$ and the category of $A(V)$-modules.  However, the representation category of a general vertex operator algebra has additional structure such as fusion rules. The  higher level Zhu algebras $A_n(V)$ in \cite{Dong-Li-Mason2} and their bimodules (\cite{Dong-Jiang}) can be used  to capture the fusion structure. 

 Another associative algebra used to study a vertex algebra was also constructed by  Zhu (\cite{Zhu}). It is the $C_2$-algebra $R(V)$, defined as $V/C_2(V)$ where $C_2(V ) = \text{Span}\{a_{-2}b \ | \ a, b \in V \}$. This space has a commutative Poisson algebra structure. When $R(V)$ is finite dimensional ($V$ is said to be $C_2$-cofinite),  $A(V)$ is also finite dimensional. Thus one expects that the pair of associative algebras $ (R(V), A(V))$ should capture more information on $V$ as well as representations of $V$, and the machinery of representation theory of finite dimensional algebras can be applied to that of $V$.  
 
\delete{A basis of the associative algebra $R(V)$ gives a finite generating set for the vertex algebra (cf. \cite{Gaberdiel-Neitzke}). A similar result giving a finite generating set for a weak $V$-module has also been obtained by Buhl in \cite{Buhl}. }

Corresponding to the Poisson algebra $R(V)$, under some finite generation assumption on $V$, $X_V=\on{Max}(R(V))$ is an affine Poisson variety which is called the associated variety and have been studied in \cite{Arakawa1}. This variety is zero dimensional if and only if the vertex algebra is $C_2$-cofinite. Another property of the associated variety is that if $V$ is the simple affine vertex operator algebra $L_{\hat{\mathfrak{g}}}(k,0)$ with $k$ admissible, then $X_V$ is the closure of a nilpotent orbit of the Lie algebra $\mathfrak{g}$ (cf. \cite{Arakawa2}). As pointed out in \cite{Feigin-Gukov}, this variety should have physical interpretations. 

In this article, we study a notion dual to that of the associated variety, namely the cohomological variety of a vertex operator algebra. Cohomological support varieties for modules over group algebras of finite groups in positive characteristic case have been studied extensively (see  \cite{Quillen,Carlson1,Carlson2} for example)  using the maximal ideal spectrum of the group cohomology ring. This theory has also be studied  for finite group schemes (\cite{Suslin-Friedlander-Bendel, Friedlander-Pevtsova}), certain finite super group schemes (\cite{Drupieski-Kujawa19, Drupieski-Kujawa21}), and even finite tensor categories (\cite{Bergh-Plavnik-Witherspoon}).  A theory of support varieties for arbitrary finite dimensional algebras was then developed using Hochschild cohomology rings (cf. \cite{Snashall-Solberg} and \cite{Solberg}). In what follows, we define the cohomological variety $X_{V, x}^{!}$ of an $\mathbb{N}$-graded vertex operator algebra $V$ at the point $\mathfrak{m}_x \in \text{Spec}(R(V)_0)$ as the maximal spectrum of a commutative algebra constructed from the Ext algebra $\on{Ext}^*_{R(V)}(\cc_x, \cc_x)$ of the one dimensional  $R(V)$-module $\cc_x \cong R(V)_0/\mathfrak{m}_x$ with $R(V)_0=R(V)/R(V)_+$. Similar to computing group cohomology, this algebra $\on{Ext}^*_{R(V)}(\cc_x, \cc_x)$ is in general quite complicated, and its direct computation is rarely fruitful unless the algebra $R(V)$ at the point $ x$ is reasonable, such as complete intersection or Gorenstein. The computations of the Ext algebras for both $A(V)$ and $ R(V)$ are carried out in \cite{Caradot-Jiang-Lin_1} for the $C_2$-cofinite triplet vertex operator algebra. It is still not clear that the Ext algebra is finitely generated.  However, the cohomological variety contains information on the Ext algebra $\on{Ext}^*_{R(V)}(\cc_x, \cc_x)$, and its study gives insight on the cohomological behaviour of $R(V)$.  Another motivation for defining the cohomological variety is to study the cohomological support variety of a representation, similar to those in \cite{Carlson1, Suslin-Friedlander-Bendel, Friedlander-Pevtsova, Friedlander-Parshall, Snashall-Solberg}. 

In Section~\ref{SectionDef} we recall the definitions of the $C_2$-algebra $R(V)$ and its associated variety $X_V$. We then define the cohomological scheme $\widetilde{X}_{V, x}^{!}$ and the cohomological variety $X_{V, x}^{!}$, and explain the duality between them and the associated scheme and associated variety. Unlike in the Hopf algebra case (such as group algebras in \cite{Carlson1, Carlson2} or restricted universal enveloping algebras in \cite{Friedlander-Parshall}) or in the Hochschild cohomology case (as in \cite{Snashall-Solberg}), the Ext algebra is not graded commutative in general. However, it contains a maximal commutative quotient which will define the cohomological scheme $ \widetilde{X}^!_{V, x}$. It should be noted that there is another way to attach a commutative graded algebra in representation theory from a non-commutative algebra with an almost commutative filtered structure as in $D$-module theory. But we take the maximal commutative quotient since the commutativity of the Ext algebras for commutative algebras fails only at some squares  of odd degree elements. This approach does not apply to $D$-module theory as the Weyl algebra does not have any finite dimensional module over fields of characteristic zero. These properties for graded commutative algebras will be studied in detail in \cite{Caradot-Jiang-Lin_2}.  In Section~\ref{SectionFunctor}, we study the functors associating $\widetilde{X}_V$, $X_V$ and $\widetilde{X}_{V, x}^{!}$ to $V$, as well as their behaviour with respect to direct sums and tensor products. Section~\ref{Sectiondimcomology} is devoted to the study of the Ext algebra $\on{Ext}^{*}_R(\cc,\cc)$ of the trivial module of an augmented commutative algebra $\epsilon: R\to \cc $, as well as determining a particular commutative quotient of the Ext algebra $\on{Ext}^{*}_R(\cc,\cc)$.  Under some assumptions on the ring, we explicitly obtain the variety $\text{Max}(\on{Ext}^{*}_R(\cc,\cc)/I)$ for a well chosen ideal $I$.  Using Tate's construction of a resolution, one can get a minimal free resolution of $ \cc $ as an $R$-module. If the algebra $ R$ is graded with presentation as a quotient of polynomial algebra, with  a minimal set of homogeneous generators, we can find a complete intersection $ \widetilde R$ such that $ R$ is quotient. Under the assumption on the degrees of the generators being at least three, we can show that the Ext algebra $\on{Ext}^{*}_R(\cc,\cc)$ has a quotient which is a polynomial algebra associated to the Ext algebra of $ \widetilde R$. This result is independent of vertex operator algebras and will be applied to $R=R(V)$ for rational affine vertex operator algebras $V$.

In Section~\ref{sec:Vir}, we determine explicitly the cohomological variety of the simple Virasoro vertex operator algebra, which is fairly easy since $R(V)$ is isomorphic to $ \cc[x]/\langle x^r\rangle $ for some $ r$. Finally, in Section~\ref{LowerboundforLiealgebras}, we obtain a lower bound on the dimension of the cohomological variety of a simple affine vertex operator algebra $L_{\hat{\mathfrak{g}}}(k,0)$ when $k$ is a nonnegative integer and $\mathfrak{g}$ is a simple finite dimensional Lie algebra. In this case, $R(L_{\hat{\mathfrak g}}(k, 0))=\on{Sym}(\mathfrak g)/\langle L((k+1)\theta)\rangle $ where $L((k+1)\theta)$ is the unique irreducible $\mathfrak g$-submodule in $ \on{Sym}^{k+1}(\mathfrak{g})$. One of the main results of this paper is:

\begin{theorem*}
Let $\mathfrak{g}$ be a finite dimensional simple Lie algebra. If $k$ is a positive integer larger or equal to $2$, then the cohomological scheme of $L_{\hat{\mathfrak{g}}}(k,0)$ contains an affine space $\cc^{N_k}$ as a closed subscheme, with
\begin{equation*}
N_k=\prod_{\alpha \in \Phi^+}((k+1)\frac{( \theta, \alpha )}{( \rho, \alpha )}+1).
\end{equation*}
\end{theorem*}

Although the associated variety of $L_{\hat{\mathfrak{g}}}(k,0)$ is a single point for all $k \in \mathbb{N}$, the cohomological variety can be very large depending on $k$, as $N_k$ is a polynomial in $k$ with positive coefficients. We will give explicit expressions of this polynomial for lower rank cases.

\section{Varieties associated to a vertex operator algebra}\label{SectionDef}
We will use the notations from \cite{Lepowsky-Li} for a vertex algebra $ (V, Y, \mathbf 1)$ and a vertex operator algebra $ (V, Y, \mathbf 1, \omega)$ as a vertex algebra together with a conformal vector $ \omega \in V_2$. The vertex operators are written as $Y(v, x)=\sum_{n\in\mathbb Z} v_n x^{-n-1}$ with $ v_n \in \on{End}_{\cc}(V)$ and $ Y(\omega, x)=\sum_{n\in \mathbb Z} L_{n}x^{-n-2}$. The $C_2$-algebra $R(V)$ for a vertex algebra is defined in \cite{Zhu} as the quotient $R(V)=V/C_2(V)$.  We note that for any $ v, u \in V$, 
\[ Y(v, x)u=\sum_{n\geq -1}v_n(u)x^{-n-1}+\sum_{n\leq -2}v_n(u)x^{-n-1}.\]
The second infinite summation has all coefficients in $ C_2(V)$.  For each element $ v\in V$, we use $ \bar v\in R(V)$ to denote its image. Then the associative multiplication in  $R(V)$ is defined by $ \bar v\bar u=\overline{v_{-1}(u)}$ and the Poisson structure is defined by $ \{ \bar v , \bar u\} =\overline{v_{0}(u)}$.  It is proved in \cite{Zhu} that $ R(V)$ is a commutative algebra over $ \cc$. We will not use the Poisson structure in this paper. 
\begin{definition}
Let $V$ be a vertex algebra.
\begin{itemize}
\item The \textbf{associated scheme} of $V$ is $\widetilde{X}_V=\on{Spec}(R(V))$, the spectrum of all prime ideals of $ R(V)$. 
\item The \textbf{associated variety} of $V$ is $X_V=\on{Max}(R(V))$, i.e., the maximal ideal spectrum, the set of all closed points in $ \on{Spec}(R(V))$. 
\end{itemize}
\end{definition}

The associated scheme and variety of a vertex operator algebra have been vastly studied and, in particular, linked to Slodowy slices in Lie algebras and W-algebras. For example, Arakawa showed in \cite{Arakawa1} that a finitely strongly generated vertex algebra $V$ is $C_2$-cofinite if and only if $\on{dim}_\cc X_V=0$, and in \cite{Arakawa2} he proved that for a Lie algebra $\mathfrak{g}$, if an element $k \in \cc$ is admissible, then $X_{L_{\hat{\mathfrak{g}}}(k,0)}$ is the closure of an orbit in the nilpotent cone of $\mathfrak{g}$.

Let $V=\sum_{n=0}^{\infty}V_n$ be an $\mathbb{N}$-graded vertex operator algebra. Here and in the rest of this article, $\mathbb{N}$-graded means  that $V_n=\on{Ker}( L(0)-n\on{Id}_V: V\to V)$ with $ n\in \mathbb N$. Then $R(V)=\bigoplus_{n=0}^{\infty}R(V)_n$ is an $\mathbb N$-graded commutative algebra with $ R(V)_n=\overline{V_n}$.  In particular, if $V$ is $C_2$-cofinite, then $ R(V)$ is a finite dimensional commutative algebra with an $\mathbb{N}$-graded structure. The commutative algebra homomorphism $ R(V)_0\to R(V)$ defines a morphism of schemes $ \on{Spec}(R(V))\to \on{Spec}(R(V)_0)$.  For each closed point $x\in \on{Max}(R(V)_0)$,  its fiber  in $ \on{Max}(R(V))$ is a conical variety with $x$ as the vertex. We are interested in studying geometric properties of the vertex in this conical variety.

\subsection{}The closed point $ x$ corresponding to the maximal ideal $\mathfrak{m}_x \in \on{Max}(R(V)_0)$ defines an irreducible $R(V)_0$-module $ \cc_x=R(V)_0/\mathfrak m_x$, which pulls back through the homomorphism $R(V)\to R(V)_0$ to an $R(V)$-module  $\cc_x$.  We can thus define the Ext algebra (also called Yoneda algebra) $\on{Ext}^{*}_{R(V)}(\cc_x,\cc_x)$, with the multiplication being given by the Yoneda product. Using the notation of Cartan \& Eilenberg (\cite{Cartan-Eilenberg}), $x$ defines an augmented algebra structure $x: R(V)\to \cc$ and we will write $H_x^*(R(V))=\on{Ext}^{*}_{R(V)}(\cc_x,\cc_x)$, the cohomology of the augmented algebra. 

Let us briefly recall that the cohomology $ H_x^*(R(V))=\on{Ext}^{*}_{R(V)}(\cc_x,\cc_x)$ can be computed by constructing a projective resolution of $ \cc_x$ as $R(V)$-module 
\[ \cdots \to P_i\to \cdots \to P_1\to P_0\to \cc_x\to 0\]
and then apply the contravariant functor $\on{Hom}_{R(V)}(-, \cc_x )$ to compute the cohomology of the resulting cochain complex.  The cohomology classes in $ H^n_x(R(V))$ can be interpreted as the equivalent classes of exact sequences (see \cite[3.4.6]{Weibel}, \cite[2.7]{Benson} and  \cite[\textrm{III}.6]{MacLane})
\[\xi: \quad0\to \cc_x  \to X_{n-1}\to \cdots \to X_0\to \cc_x\to 0\]
 of $R(V)$-modules. Given another exact sequence
 \[\eta:\quad 0\to \cc_x  \to Y_{m-1}\to \cdots \to Y_0\to \cc_x\to 0,\]
 the Yoneda product $ \xi\circ \eta $ is defined by simply connecting the exact sequences 
\begin{align}\label{eq:Yoneda-composition}\xi\circ \eta: \quad 0\to \cc_x  \to X_{n-1}\to \cdots \to X_0\to Y_{m-1}\to \cdots\to Y_0\to \cc_x\to 0
 \end{align}
 with $ X_0\to Y_{m-1}$ being the obvious composition map.  As discussed in \cite{Caradot-Jiang-Lin_1}, the Yoneda algebra can be defined for any object in any abelian category. The algebra $\on{Ext}^{*}_{R(V)}(\cc_x, \cc_x)$ is graded but is not graded commutative in general. 

Ideally, we would want to look at the prime and maximal spectrums of $H_x^*(R(V))$ and compare them to $\widetilde{X}_V$ and $X_V$. But $H_x^*(R(V))$ might not be commutative. Even when $R(V)$ is a Hopf algebra, $H_x^*(R(V))$ is graded commutative and not commutative. So we need to construct a commutative algebra. 

We define the set 
\[ \mathcal{I}_x= \{ I \text{ two sided ideal in } H_x^{*}(R(V)) \text{ such that } H_x^{*}(R(V))/I \text{ is commutative}\} \]
and a partial order $\succcurlyeq$ on $\mathcal{I}_x$ with $I \succcurlyeq J$ if and only if $I \subseteq J$. When $I \succcurlyeq J$, we denote $\varphi_{I,J}: H_x^{*}(R(V))/I \longrightarrow H_x^{*}(R(V))/J$ the natural projection. If $I \succcurlyeq J \succcurlyeq K$, then the following diagram commutes:
 \begin{center}
  \begin{tikzpicture}[scale=0.9,  transform shape]
  \tikzset{>=stealth}
  
\node (1) at ( -3,0){$H_x^{*}(R(V))/I$};
\node (2) at ( 3,0){$H_x^{*}(R(V))/J$};
\node (3) at ( 0,-2){$H_x^{*}(R(V))/K$};

\node (4) at ( 0,-0.9){$\circlearrowleft$};

\draw [decoration={markings,mark=at position 1 with
    {\arrow[scale=1.2,>=stealth]{>}}},postaction={decorate}] (1) --  (2) node[midway, above] {$\varphi_{I,J}$};
\draw [decoration={markings,mark=at position 1 with
    {\arrow[scale=1.2,>=stealth]{>}}},postaction={decorate}] (2)  --  (3) node[midway, below right] {$\varphi_{J, K}$};

\draw [decoration={markings,mark=at position 1 with
    {\arrow[scale=1.2,>=stealth]{>}}},postaction={decorate}] (1)  --  (3) node[midway, below left] {$\varphi_{I,K}$};
\end{tikzpicture}
 \end{center}
and $\varphi_{I,I}$ is the identity on $H_x^{*}(R(V))/I$. Therefore $\{H_x^{*}(R(V))/I, \varphi_{I, J}, \mathcal{I}_x\}$ is an inverse system, and we can define the inverse limit:
\begin{align*}
H_x^{*}(R(V))^{ab}= \underset{I \in \mathcal{I}_x}{\varprojlim} \ H_x^{*}(R(V))/I = \left \{a \in \prod_{I \in \mathcal{I}_x} H_x^{*}(R(V))/I \ | \ a_J = \varphi_{I, J}(a_I)\right \}.
\end{align*}
 
We have a natural morphism of algebras:
\begin{align*}
\begin{array}{cccc}
c_{R(V),x}:&  H_x^{*}(R(V)) & \longrightarrow & H_x^{*}(R(V))^{ab} \\
 & a  & \longmapsto & (a \text{ mod }I)_{I \in \mathcal{I}_x}.
\end{array}
\end{align*}
For each $ I \in \mathcal I_x$, let $ p_I: H_x^{*}(R(V))^{ab}\to H_x^{*}(R(V))/I$ be the projection and $ \pi_I: H_x^{*}(R(V))\to H_x^{*}(R(V))/I$ be the quotient map. Following the construction, we have $ \pi_I=p_I\circ c_{R(V), x}$. Note that $ \pi_I$ is surjective, so is $ p_I$.  

In fact, the commutative algebra $H_x^{*}(R(V))^{ab}$ has the following universal property: for any commutative algebra $A$ and morphism of algebras $H_x^{*}(R(V)) \longrightarrow A$, there exists a unique morphism of algebras $H_x^{*}(R(V))^{ab} \longrightarrow A$ which makes the diagram commute:
 \begin{center}
  \begin{tikzpicture}[scale=0.9,  transform shape]
  \tikzset{>=stealth}
  
\node (1) at ( 0,0){$H_x^{*}(R(V))$};
\node (2) at ( 5,0){$H_x^{*}(R(V))^{ab}$};
\node (3) at ( 0,-3){$A$};

\draw [decoration={markings,mark=at position 1 with
    {\arrow[scale=1.2,>=stealth]{>}}},postaction={decorate}] (1) --  (2) node[midway, above] {$c_{R(V), x}$};
\draw [decoration={markings,mark=at position 1 with
    {\arrow[scale=1.2,>=stealth]{>}}},postaction={decorate}] (2)  --  (3);
\draw [decoration={markings,mark=at position 1 with
    {\arrow[scale=1.2,>=stealth]{>}}},postaction={decorate}] (1)  --  (3) node[midway, left] {};
\end{tikzpicture}.
 \end{center}

We remark that for a general  non-commutative algebra $ U$, one can follow exactly the same construction to define $ U^{ab}$ as the universal commutative algebra satisfying the above universal property. It is necessary that $ U^{ab}=U/\langle [U,U]\rangle $ since the subset $ [U,U]=\{ [u,v]=uv-vu\;|\; u, v\in U\}$ has to be in the kernel of any algebra homomorphism $ U\to A$ with $ A$ being commutative. However, it is likely that $ U^{ab}=\{0\}$. For example, let $ U=\cc\langle x, y\rangle/\langle [x,y]-1\rangle$ be the Weyl algebra. Then $ U^{ab}=0$. If $ U=U(\mathfrak g)$ is the universal enveloping algebra of a semisimple Lie algebra, then $ U^{ab}=\cc$. Thus the algebra $U^{ab}$ is not a good model for attaching algebraic geometric objects.  But for the Yoneda algebra $ H^*_x(R)$ of a commutative algebra $R$, the commutative algebra $ H_x^{*}(R)^{ab}$ is close enough to $ H^*_x(R)$, in particular when $R$ is a complete intersection. This will be discussed in \cite{Caradot-Jiang-Lin_2}. 

Since $H_x^{*}(R(V))^{ab}$ is a commutative algebra, it makes sense to define its spectra. 
\begin{definition}\label{def:cohom_var}
Let $V$ be an $\mathbb{N}$-graded vertex operator algebra and set $\mathfrak{m}_x \in \on{Max}(R(V)_0)$.
\begin{itemize}\setlength{\itemsep}{5pt}
\item The \textbf{cohomological scheme} of $V$ at $x$ is $\widetilde{X}_{V, x}^{!}=\on{Spec}(H_x^{*}(R(V))^{ab})$.
\item The \textbf{cohomological variety} of $V$ at $x$ is $X_{V ,x}^{!}=\on{Max}(H_x^{*}(R(V))^{ab})$.
\end{itemize}
\end{definition}
Note that we have been unspecific whether the algebra $ H_x^{*}(R(V))^{ab}$ is finitely generated over $\cc$ or not. Thus the variety may not necessarily be of finite type in general. 

\delete{\begin{remark}
An equivalent way to define the cohomological scheme/variety would be to look at the prime/maximal spectrum of $H_x^{*}(R(V))/[H_x^{*}(R(V)), H_x^{*}(R(V))]$, the largest commutative quotient of $H_x^{*}(R(V))$. Indeed, any ideal in $\mathcal{I}_x$ contains the commutator ideal $[H_x^{*}(R(V)), H_x^{*}(R(V))]$, and so the inverse limit is entirely defined by the above quotient. The advantage of our definition is that if we find some ideal in $\mathcal{I}_x$, then we can determine a subvariety of the cohomological variety (cf. Proposition~\ref{DecompXcohom}).
\end{remark}}

For a positively graded ring $A=\bigoplus_{i \geq 0}A_i$ we may consider $A_0$ as a left $A$-module through $A_0 = A/A_{>0}$ and form the positively graded ring $E(A) = \on{Ext}^{*}_{A}(A_0,A_0)$. If a Koszul ring $A$ satisfies a suitable finiteness condition, then the algebra $E(A)$ is Koszul as well and there exists a canonical isomorphism $E(E(A)) \cong A$. This motivates the appellation of $E(A)$ as the Koszul dual of $A$ (cf. \cite{Beilinson-Ginzburg-Soergel} and \cite{Green-Reiten-Solberg} for definitions and results on Koszul algebras). In our context, $A=R(V)$ is not necessarily Koszul but it is $\mathbb{N}$-graded and, if $V_0=\cc \mathbf{1}$, then $A_0=\cc$ and so $H_x^*(R(V))=E(R(V))$ with $\mathfrak{m}_x \in \on{Max}(R(V)_0)$. Therefore there is a duality between $R(V)$ and $H_x^*(R(V))$, which suggests a dual relation between $\widetilde{X}_V$, $X_V$ and $\widetilde{X}^{!}_{V, x}$, $X^{!}_{V, x}$.

The algebra $H_x^{*}(R(V))^{ab}$ is not easy to determine in the general case. However, we can decompose its spectrum as a union of closed subsets.

\begin{proposition}\label{DecompXcohom}
For an $\mathbb{N}$-graded vertex operator algebra $V$ with $\mathfrak{m}_x \in \on{Max}(R(V)_0)$, we have:
\begin{align*}
X_{V, x}^{!}=\displaystyle  \bigcup_{I \in \mathcal{I}_x} \on{Max}(H_x^{*}(R(V))/I).
\end{align*}
\end{proposition}

\begin{proof}  For each $ I\in \mathcal I_x$, since the projection map $ p_I: H_x^{*}(R(V))^{ab}\to H_x^{*}(R(V))/I $ is surjective,  for any  maximal ideal $M$ of $H_x^{*}(R(V))/I$, we have 
$p_I^{-1}(M)\in \on{Max}(H_x^{*}(R(V))^{ab})$ and $\on{Max}(H^*_x(R(V))/I)$ is the closed subvariety  defined by the ideal $ \on{Ker} (p_I)$ of $ H_x^{*}(R(V))^{ab})$.
\delete{For any $I' \in \mathcal{I}_x$, we have $L \subseteq I'$ with $L=[H_x^{*}(R(V)), H_x^{*}(R(V))]$. We thus have a surjective morphism $\varphi_{L,I'}: H_x^{*}(R(V))/L \longrightarrow H_x^{*}(R(V))/I'$. In particular, the ideal $\mathfrak{m}=\varphi_{L, I}^{-1}(M)$ is maximal in $H_x^{*}(R(V))/L$ with the quotient $(H_x^{*}(R(V))/L)/\mathfrak{m}$ being a field $\mathbf{k}$. Therefore for any $a \in H_x^{*}(R(V))^{ab}$, we have $\overline{a} \in H_x^{*}(R(V))^{ab}/M^I$ given by $(\overline{a_L}, \overline{\varphi_{L, I'}}(\overline{a_L}) \ | \ I' \in \mathcal{I}_x)$, with $\overline{a_L}=a_L \on{mod} \mathfrak{m}$ and $\overline{\varphi_{L, I'}}: (H_x^{*}(R(V))/L)/\mathfrak{m} \longrightarrow (H_x^{*}(R(V))/I')/\varphi_{L, I'}(\mathfrak{m})$ the induced surjective morphism. However, we notice that $\varphi_{L, I'}(\mathfrak{m})$ is either a maximal ideal of $H_x^{*}(R(V))/I'$ or the whole of $H_x^{*}(R(V))/I'$, and thus the morphism $\overline{\varphi_{L, I'}}$ is either $0$ or an isomorphism $\mathbf{k} \longrightarrow \mathbf{k}$. Hence $\overline{a}$ is entirely defined by $\overline{a_L} \in \mathbf{k}$. It follows that $H_x^{*}(R(V))^{ab}/M^I \cong \mathbf{k}$ and $M^I$ is a maximal ideal of $H_x^{*}(R(V))^{ab}$.}

Let $N$ be a maximal ideal of $H_x^{*}(R(V))^{ab}$.  Consider the composition $\sigma=\psi_N\circ c_{R(V),x}$ 
\[ \xymatrix{ H^*_x(R(V)) \ar[r]^{c_{R(V),x}} \ar[dr]_{\sigma}& H_x^{*}(R(V))^{ab}\ar[d]^{\psi_N}\\
& H_x^{*}(R(V))^{ab}/N}
\] of algebra homomorphisms. Let $ I=\on{Ker}(\sigma)\subset H^*_x(R(V))$ be a kernel of this homomorphism. Since $ H^*_x(R(V))/I=\on{Im}(\sigma)$ is subring of a field, we have $ I\in \mathcal I_x$ and we have the following commutative diagram
\[ \xymatrix{ H^*_x(R(V))\ar[r]^{c_{R(V),x}}\ar[d]_{\pi_I}& H_x^{*}(R(V))^{ab}\ar[d]^{\psi_N}\ar[dl]_{p_I}\\
H^*_x(R(V))/I\ar[r]_{\iota}& H_x^{*}(R(V))^{ab}/N}
\]
with $ \iota: \on{Im}(\sigma)\to H_x^{*}(R(V))^{ab}/N$ being the embedding and $\sigma=\iota\circ \pi_I$.  By the uniqueness of the map $\psi_N$, we have $ \psi_N=\iota\circ p_I$. The surjectivity of the quotient map $ \psi_N$ implies that $ \iota$ must be identity, and so $H_x^{*}(R(V))/I$ is a field and $p_I=\psi_N$.  Thus $ N=p_I^{-1}(0)\in  \on{Max}(H_x^*(R(V))/I)$. 
\end{proof}
\delete{
Otherwise, $\pi_I(N)=H_x^*(R(V))/I$ for all $I \in \mathcal I_x$. 
 the $L$-th coordinate $N_L$ is a maximal ideal of $H_x^{*}(R(V))/L$. Indeed, if it were the whole of $H_x^{*}(R(V))/L$, then $\varphi_{L, I}(N_L)=H_x^{*}(R(V))/I$ for any $I \in \mathcal{I}_x$. But then in the quotient 
$H_x^{*}(R(V))^{ab}/N$, each coordinate would be zero, which would lead to $H_x^{*}(R(V))^{ab}/N=\{0\}$, contradicting the maximality of $N$. Thus $N_L$ is a proper ideal. As $H_x^{*}(R(V))^{ab}/N$ is a field, each coordinate has to be either a field or zero. Hence $(H_x^{*}(R(V))/L)/N_L$ is a field, and $N_L$ is a maximal ideal. We see that $N \subseteq (N_L)^L$ and they are both maximal. Therefore $N = (N_L)^L$, and any maximal ideal of $H_x^{*}(R(V))^{ab}$ comes from a maximal ideal of $H_x^{*}(R(V))/L$. As $L \in \mathcal{I}_x$, the statement is proved.
\end{proof}}

\section{Properties of the cohomological schemes and varieties}\label{SectionFunctor}
\subsection{Functoriality}

Let $\mathcal{V}$ be the category of vertex algebras, $\mathcal{V}_{f.g}$ the category of vertex algebras with finitely generated $C_2$-algebras, $\mathcal{VOA}_{f.g}$ the category of $\mathbb{N}$-graded vertex operator algebras with finitely generated $C_2$-algebras,  $\mathcal{T}$ the category of algebraic varieties, and $\mathcal{S}$ the category of affine schemes. We define the following maps:
\begin{align*}
\renewcommand\arraystretch{1.5}
\begin{array}{l}
\begin{array}{cccc}
\widetilde{F}: & \mathcal{V} & \longrightarrow & \mathcal{S} \\
 &V & \longmapsto &\widetilde{X}_V, 
\end{array} \\
\begin{array}{cccc}
F: & \mathcal{V}_{f.g} & \longrightarrow & \mathcal{T} \\
 &V & \longmapsto &X_V ,
\end{array} \\
\begin{array}{cccc}
\widetilde{F}^{!}: & \mathcal{VOA}_{f.g} & \longrightarrow & \mathcal{S} \\
 &V & \longmapsto & \widetilde{X}^!_{V ,x}.
\end{array}
\end{array}
\end{align*}

We prove the following:
\begin{proposition}
The maps $F$ and $\widetilde{F}$ are contravariant functors, and the map $\widetilde{F}^{!}$ is a covariant functor.
\end{proposition}

\begin{proof} We note that the assignment $V\mapsto R(V)$ is covariant functor from the category of vertex algebras to the category of commutative $\cc$-algebras, and $\on{Spec}$ is a contravariant functor from the category of commutative $\cc$-algebras to the category of schemes. Hence $\widetilde{F}$ is contravariant. One notes that taking the maximal ideal spectrum $ \on{Max}(R(V))$ is not functorial on the category of commutative $\cc$-algebras, but it is functorial on the category of finitely generated commutative $\cc$-algebras  since,  by Zariski's Lemma or Hilbert Nullstellensatz, every maximal ideal is a kernel of an algebra homomorphism $ R\to \cc$. 

Given a homomorphism of  augmented algebras  $ \phi: (R\to \cc_x)\to (R'\to \cc_{x'})$, by \cite[VIII. 3]{Cartan-Eilenberg} and noting that $ \cc_{x'}$ restricted to $ R$ is isomorphic to $ \cc_x$,  there is a natural homomorphism of  graded $\cc$-vector spaces 
$\phi^\#: H^*(R')\to H^*(R)$.  This homomorphism preserves the Yoneda product since the pullback functor $\phi^*: R'\Mod\to R\Mod$ is exact.  Clearly, the assignment $A\mapsto A^{ab}$ is also covariant. Thus $ (R\to \cc_x)\mapsto \on{Spec}(H_x^{*}(R)^{ab})$ is a covariant functor.  We remark that we cannot say that $ (R\to \cc_x)\mapsto \on{Max}(H_x^{*}(R)^{ab})$ is functorial since it is not known when $H_x^{*}(R)^{ab}$ is a finitely generated algebra.
\end{proof}
\delete{
Let $f:V \longrightarrow V'$ be a homomorphism of vertex algebras. The space $C_2(V)$ is spanned by $u_{-2}v$ with $u,v \in V$, and we have $f(u_{-2}v)=f(u)_{-2}f(v)$, hence $f(C_2(V)) \subseteq C_2(V')$. Therefore $f$ induces a morphism of Poisson algebras $\varphi : R(V) \longrightarrow R(V')$. By taking the spectrum on both sides, we obtain a continuous morphism $\widetilde{F}(f)=\varphi^*: \widetilde{X}_{V'} \longrightarrow \widetilde{X}_{V}$. If $R(V')$ is finitely generated, then we can look at the maximal spectrums and get $F(f): X_{V'} \longrightarrow X_{V}$. 

It is immediate to check that $\widetilde{F}(\on{id}_{V})=\on{id}_{\widetilde{X}_V}$ and $F(\on{id}_{V})=\on{id}_{X_V}$. 

Let $f_1:V \longrightarrow V'$ and $f_2:V' \longrightarrow V''$ be two vertex algebra morphisms. The composition $f_2 \circ f_1$ induces a morphism $\varphi_{1,2}=\varphi_2 \circ \varphi_1 : R(V) \longrightarrow R(V'')$, which then leads to $\varphi_{1,2}^*:\widetilde{X}_{V''} \longrightarrow \widetilde{X}_{V}$. If $\mathfrak{p} \in \widetilde{X}_{V''}$, then $\varphi_{1,2}^*(\mathfrak{p})=\varphi_{1,2}^{-1}(\mathfrak{p})=(\varphi_2 \circ \varphi_1)^{-1}(\mathfrak{p})=\varphi_1^{-1} \circ \varphi_2^{-1}(\mathfrak{p})=\varphi_1^{*} \circ \varphi_2^{*}(\mathfrak{p})$. We have then $\widetilde{F}(f_1) \circ \widetilde{F}(f_2)=\widetilde{F}(f_2 \circ f_1)$ and so $\widetilde{F}$ is a contravariant functor. It works the same for $F$.

Let $f:V \longrightarrow V'$ be a morphism of $\mathbb{N}$-graded vertex operator algebras with finitely generated $C_2$-algebras. As above, it induces $\varphi : R(V) \longrightarrow R(V')$. By setting $\mathfrak{m}_x \in \on{Max}(R(V')_0)$, we can define the $R(V')$-module $\cc_x$ and the algebra $H_x^{*}(R(V'))$. As the projection $R(V') \longrightarrow R(V')_0$ is surjective, the preimage of $\mathfrak{m}_x$ is a maximal ideal $M_x$ is $R(V')$ containing $R(V')_{>0}$, and because $R(V')$ is finitely generated, the ideal $\varphi^{-1}(M_x)$ is maximal in $R(V)$. The morphism $\varphi$ preserves the degree because it is induced by a homomorphism of vertex operator algebras which preserves the conformal vector $\omega$, therefore $\varphi(R(V)_{>0}) \subseteq R(V')_{>0} \subseteq M_x$, and so $R(V)_{>0} \subseteq \varphi^{-1}(M_x)$. We thus have $\cc \cong R(V)/\varphi^{-1}(M_x) \cong (R(V)/R(V)_{>0})/(\varphi^{-1}(M_x)/R(V)_{>0}) \cong R(V)_0/(\varphi^{-1}(M_x)/R(V)_{>0})$. We see that $\varphi^{-1}(M_x)/R(V)_{>0}$ is a maximal ideal $\mathfrak{m}_y$ of $R(V)_0$. 

The action of $R(V)$ on $\cc$ is given by the composition $R(V) \stackrel{\varphi}{\longrightarrow} R(V') \longrightarrow R(V')_0 \stackrel{/ \mathfrak{m}_x}{\longrightarrow} \cc_x$, which is the same as $R(V) \longrightarrow R(V)_0 \stackrel{/ \mathfrak{m}_y}{\longrightarrow} \cc_y$. Therefore $\varphi$ induces an algebra morphism:
\begin{align*}
\varphi^{\#} : H_x^*(R(V')) \longrightarrow H_y^*(R(V')),
\end{align*}
and we have the following commutative diagram:
 \begin{center}
  \begin{tikzpicture}[scale=0.9,  transform shape]
  \tikzset{>=stealth}
  
\node (1) at ( 0,0){$H_x^{*}(R(V'))$};
\node (2) at ( 5,0){$H_x^{*}(R(V'))^{ab}$};
\node (3) at ( 0,-3){$H_y^{*}(R(V))^{ab}$};

\node (4) at ( 1.5,-1){$\circlearrowleft$};

\draw [decoration={markings,mark=at position 1 with
    {\arrow[scale=1.2,>=stealth]{>}}},postaction={decorate}] (1) --  (2) node[midway, above] {$c_{R(V'), x}$};
\draw [decoration={markings,mark=at position 1 with
    {\arrow[scale=1.2,>=stealth]{>}}},postaction={decorate}] (2)  --  (3) node[midway, below right] {$(\varphi^{\#})^c$};
\draw [decoration={markings,mark=at position 1 with
    {\arrow[scale=1.2,>=stealth]{>}}},postaction={decorate}] (1)  --  (3) node[midway, left] {$c_{R(V), y} \circ \varphi^{\#}$};
\end{tikzpicture}
 \end{center}
The existence of $(\varphi^{\#})^c$ comes from the commutativity of $H_y^{*}(R(V_1))^{ab}$ and the universal property of $H_x^{*}(R(V_2))^{ab}$. This morphism then induces a continuous morphism $\widetilde{F}^{!}(f): \widetilde{X}_{V_1, y}^{!} \longrightarrow \widetilde{X}_{V_2, x}^{!}$, and we can verify that $\widetilde{F}^{!}(\on{id}_{V})=\on{id}_{\widetilde{X}_{V, x}^{!}}$. 

Let $f_1:V \longrightarrow V'$ and $f_2:V' \longrightarrow V''$ be two morphisms of $\mathbb{N}$-graded vertex operator algebras with finitely generated $C_2$-algebras. Set $\mathfrak{m}_x \in \on{Max}(R(V'')_0)$. The morphism $\varphi_{1,2}$ induces $\varphi_{1,2}^{\#} : H_x^*(R(V'')) \longrightarrow H_z^*(R(V))$ with $\mathfrak{m}_z=(\varphi_2 \circ \varphi_1)^{-1}(M_x)/R(V)_{>0}$, and one can check that $\varphi_{1,2}^{\#}=\varphi_{1}^{\#} \circ \varphi_{2}^{\#}$. Indeed, if $S \in H_x^n(R(V''))$ is an $n$-fold extension of $R(V_3)$-modules, then $\varphi_{1,2}^{\#}(S)$ is the same extension seen as an extension of $R(V)$-modules through $\varphi_{1,2}$. We thus have $\varphi_{1, 2}^{\#}=\varphi_{1}^{\#} \circ \varphi_{2}^{\#}: H_x^{*}(R(V'')) \longrightarrow H_z^{*}(R(V))$, which leads to:  \begin{center}
  \begin{tikzpicture}[scale=0.9,  transform shape]
  \tikzset{>=stealth}
  
\node (1) at ( 0,0){$H_x^{*}(R(V''))$};
\node (2) at ( 5,0){$H_x^{*}(R(V''))^{ab}$};
\node (3) at ( 0,-2){$H_y^{*}(R(V'))$};
\node (4) at ( 0,-4){$H_z^{*}(R(V))$};
\node (5) at ( 0,-6){$H_z^{*}(R(V))^{ab}$};

\node (6) at ( 2,-2){$\circlearrowleft$};

\draw [decoration={markings,mark=at position 1 with
    {\arrow[scale=1.2,>=stealth]{>}}},postaction={decorate}] (1) --  (2) node[midway, above] {$c_{R(V_3), x}$};
\draw [decoration={markings,mark=at position 1 with
    {\arrow[scale=1.2,>=stealth]{>}}},postaction={decorate}] (2)  --  (5) node[midway, below right] {$(\varphi_1^{\#} \circ \varphi_2^{\#})^c$};
\draw [decoration={markings,mark=at position 1 with
    {\arrow[scale=1.2,>=stealth]{>}}},postaction={decorate}] (1)  --  (3) node[midway, left] {$\varphi_2^{\#}$};
    \draw [decoration={markings,mark=at position 1 with
    {\arrow[scale=1.2,>=stealth]{>}}},postaction={decorate}] (3)  --  (4) node[midway, left] {$ \varphi_1^{\#}$};
    \draw [decoration={markings,mark=at position 1 with
    {\arrow[scale=1.2,>=stealth]{>}}},postaction={decorate}] (4)  --  (5) node[midway, left] {$c_{R(V), z}$};
\end{tikzpicture}
 \end{center}
 where $\mathfrak{m}_y=\varphi_2^{-1}(M_x)/R(V')_{>0}$. We define $\mathcal{I}_{3, x}$ (respectively $\mathcal{I}_{2, y}$, $\mathcal{I}_{1, z}$) the set of ideals defining the inverse system for $H_x^{*}(R(V''))$ (respectively $H_y^{*}(R(V'))$, $H_z^{*}(R(V))$). As the above diagram commutes, for all $a \in H_x^{*}(R(V''))$, we obtain:
\begin{align*}
(\varphi_1^{\#} \circ \varphi_2^{\#})^c((a \text{ mod } J)_{J \in \mathcal{I}_{3, x}})=\left ( (\varphi_1^{\#} \circ \varphi_2^{\#})(a) \text{ mod } I \right )_{I \in \mathcal{I}_{1, z}}.
\end{align*}
By computing the composition of $(\varphi_1^{\#})^c$ and $(\varphi_2^{\#})^c$, we can determine that $(\varphi_1^{\#} \circ \varphi_2^{\#})^c=(\varphi_1^{\#})^c \circ (\varphi_2^{\#})^c$ on $\on{Im}(c_{R(V_3), x})$. 

Set $(a_I)_{I \in \mathcal{I}_{3, x}} \in H_x^{*}(R(V''))^{ab}$ and $r_3 \in H_x^{*}(R(V''))$ such that $a_{\text{Ker} (c_{R(V), z} \circ \varphi_1^{\#} \circ \varphi_2^{\#})}=r_3$ mod $\text{Ker} (c_{R(V), z} \circ \varphi_1^{\#} \circ \varphi_2^{\#})$. As $(\varphi_1^{\#} \circ \varphi_2^{\#})^c$ is the projection on the $\text{Ker} (c_{R(V), z} \circ \varphi_1^{\#} \circ \varphi_2^{\#})$-component, we get $(\varphi_1^{\#} \circ \varphi_2^{\#})^c(a)=a_{\text{Ker} (c_{R(V), z} \circ \varphi_1^{\#} \circ \varphi_2^{\#})}=(\varphi_1^{\#} \circ \varphi_2^{\#})^c(c_{R(V''), x}(r_3))$. 

Moreover, there exists $r_3' \in H_x^{*}(R(V''))$ such that $a_{\text{Ker} (c_{R(V'), y} \circ \varphi_2^{\#})}=r_3' \text{ mod } \text{Ker} (c_{R(V_2), y} \circ \varphi_2^{\#})$. As $(\varphi_2^{\#})^c$ is the projection on the $\text{Ker} (c_{R(V'), y} \circ \varphi_2^{\#})$-component, we see that $(\varphi_2^{\#})^c(a)=( \varphi_2^{\#})^c(c_{R(V''), x}(r_3'))$, and so $(\varphi_1^{\#})^c \circ (\varphi_2^{\#})^c(a)=(\varphi_1^{\#})^c \circ (\varphi_2^{\#})^c(c_{R(V''), x}(r_3'))$. 

Because $(\varphi_1^{\#})^c \circ c_{R(V'), y}= c_{R(V_1), z} \circ \varphi_1^{\#}$, we see that $(\varphi_1^{\#})^c \circ c_{R(V'), y} \circ \varphi_2^{\#}= c_{R(V), z} \circ \varphi_1^{\#} \circ \varphi_2^{\#}$, and thus $\text{Ker}(c_{R(V'), y} \circ \varphi_2^{\#}) \subseteq \text{Ker}(c_{R(V), z} \circ \varphi_1^{\#} \circ \varphi_2^{\#})$. It follows that $\text{Ker}(c_{R(V''), y} \circ \varphi_2^{\#}) \succcurlyeq \text{Ker}(c_{R(V), z} \circ \varphi_1^{\#} \circ \varphi_2^{\#})$ in $\mathcal{I}_{3, x}$, which gives a surjective transition map:
\begin{align*}
\psi : H_x^{*}(R(V''))/\text{Ker}(c_{R(V'), y} \circ \varphi_2^{\#}) \longrightarrow H_x^{*}(R(V''))/\text{Ker}(c_{R(V), z} \circ \varphi_1^{\#} \circ \varphi_2^{\#})
\end{align*}
coming from the inverse system structure. As $a \in H_x^{*}(R(V''))^{ab}$, it follows that $\psi$ sends $a_{\text{Ker} (c_{R(V'), y} \circ \varphi_2^{\#})}$ to $a_{\text{Ker} (c_{R(V), z} \circ \varphi_1^{\#} \circ \varphi_2^{\#})}$, meaning that $r'_3 - r_3 \in \text{Ker} (c_{R(V), z} \circ \varphi_1^{\#} \circ \varphi_2^{\#})$. Based on the previously done computations, we get that $(\varphi_1^{\#} \circ \varphi_2^{\#})^c(c_{R(V''), x}(r_3))=(\varphi_1^{\#} \circ \varphi_2^{\#})^c(c_{R(V''), x}(r'_3))=(\varphi_1^{\#})^c \circ (\varphi_2^{\#})^c(c_{R(V''), x}(r'_3))$, and so $(\varphi_1^{\#} \circ \varphi_2^{\#})^c(a)=(\varphi_1^{\#})^c \circ (\varphi_2^{\#})^c(a)$. We have thus proved that:
\begin{align*}
(\varphi_1^{\#} \circ \varphi_2^{\#})^c=(\varphi_1^{\#})^c \circ (\varphi_2^{\#})^c,
\end{align*}
and so $\widetilde{F}^{!}(f_2 \circ f_1)=(((\varphi_2 \circ \varphi_1)^{\#})^c)^*=((\varphi_1^{\#} \circ \varphi_2^{\#})^c)^*=((\varphi_1^{\#})^c \circ (\varphi_2^{\#})^c)^*=((\varphi_2^{\#})^c)^* \circ ((\varphi_1^{\#})^c)^*=\widetilde{F}^{!}(f_2) \circ \widetilde{F}^{!}(f_1)$. Therefore $\widetilde{F}^{!}$ is a covariant functor.
\end{proof}}

\delete{\begin{remark}
The map $V \longrightarrow X_{V, x}^{!}$ might not be functorial because a morphism $ X_{V_1, y}^{!} \longrightarrow  X_{V_2, x}^{!}$ would be induced by $H_x^{*}(R(V_2))^{ab} \longrightarrow H_y^{*}(R(V_1))^{ab}$, and the preimage of a maximal ideal might not be maximal.
\end{remark}}

\subsection{The structure preserving properties of the functors}

There are two naturally defined operations on vertex algebras, namely the direct sum (rather direct product \cite{Frenkel-Huang-Lepowsky} )and the tensor product (cf. \cite[3.12]{Lepowsky-Li}). The next result shows how the previous functors behave with regards to these operations.

\begin{theorem} \label{XvandXv!}
Let $V_1, V_2$ be vertex algebras in $\mathcal{V}$, $\mathcal{V}_{f.g}$ or $\mathcal{VOA}_{f.g}$ such that the following objects are well-defined. Set $\mathfrak{m}_{x_1} \in \on{Spec}(R(V_1)_0)$ and $\mathfrak{m}_{x_2} \in \on{Spec}(R(V_2)_0)$. We have the following relations:
\[ \renewcommand\arraystretch{1.5}
\begin{array}{ll}
\bullet \widetilde{X}_{V_1 \oplus V_2} \cong \widetilde{X}_{V_1} \sqcup \widetilde{X}_{V_2}, & \bullet \widetilde{X}_{V_1 \otimes V_2} \cong \widetilde{X}_{V_1} \times \widetilde{X}_{V_2}, \\[5pt]
\bullet X_{V_1 \oplus V_2} \cong X_{V_1} \sqcup X_{V_2}, &  \bullet X_{V_1 \otimes V_2} \cong X_{V_1} \times X_{V_2}, \\[5pt]
\bullet \widetilde{X}^{!}_{V_1 \oplus V_2, x} \cong \left\{\begin{array}{c}
        \widetilde{X}^{!}_{V_1, x_1}  \text{ if } x=x_1,\\
        \widetilde{X}^{!}_{V_2, x_2}  \text{ if } x=x_2,
        \end{array}\right. & \bullet \widetilde{X}^{!}_{V_1 \otimes V_2, (x_1, x_2)} \cong \widetilde{X}^{!}_{V_1, x_1} \times \widetilde{X}^{!}_{V_2, x_2}   \text{ if } R(V_1), R(V_2) \text{ local}, \\[5pt]
\bullet X^{!}_{V_1 \oplus V_2 , x} \cong \left\{\begin{array}{c}
        X^{!}_{V_1, x_1}  \text{ if } x=x_1,\\
        X^{!}_{V_2, x_2}  \text{ if } x=x_2,
        \end{array}\right. & \bullet X^{!}_{V_1 \otimes V_2 , (x_1, x_2)} \cong X^{!}_{V_1, x_1} \times X^{!}_{V_2, x_2}  \text{ if }R(V_1), R(V_2)\text{ local}.
\end{array}
\]
\end{theorem}

We will need the following lemma:

\begin{lemma}\label{lem:R(V)tensor}
For vertex algebras $V_1$ and $V_2$, we have
\[R(V_1 \oplus V_2) \cong R(V_1) \oplus R(V_2) \quad \text{and} \quad  R(V_1 \otimes V_2) \cong R(V_1) \otimes R(V_2). \]
\end{lemma}

\begin{proof}
Let $V_1$, $V_2$ be vertex algebras. By definition of the direct sum of vertex algebras, for $a, b \in V_1$ and $c, d \in V_2$, we have $(a, c)_{-2}(b, d)=(a_{-2}b, c_{-2}d)$. Therefore we have $C_2(V_1 \oplus V_2)=C_2(V_1) \oplus C_2(V_2)$. There is thus an isomorphism of vector spaces between $R(V_1 \oplus V_2)$ and $R(V_1) \oplus R(V_2)$. It is straightforward to check that it is an isomorphism of algebras.

The tensor product $V_1 \otimes V_2$ is a vertex algebra with $\mathbf{1}=\mathbf{1}_1 \otimes \mathbf{1}_2$ and for $v_i \in V_i$, we have $(v_1 \otimes v_2)_n=\sum_{k \in \mathbb{Z}}(v_1)_k \otimes (v_2)_{n-k-1}$. It follows that for $a \otimes b \in V_1 \otimes V_2$, we have
\[ (v_1 \otimes v_2)_{-2}(a \otimes b)=\sum_{k \in \mathbb{Z}}(v_1)_ka \otimes (v_2)_{-k-3}b .\]
But if $k \geq -1$, then $(v_2)_{-k-3}b \in C_2(V_2)$, and if $k \leq -2$, then $(v_1)_k a \in C_2(V_1)$. Hence $C_2(V_1 \otimes V_2) \subseteq C_2(V_1) \otimes V_2+ V_1 \otimes C_2(V_2)$. 

If $a, b \in V_1$ and $c, d \in V_2$, we have
\begin{align*}
\begin{array}{l}
a_{-2}b \otimes c =(a \otimes \mathbf{1}_{2})_{-2}(b \otimes c) \in C_2(V_1 \otimes V_2), \\[5pt]
b \otimes c_{-2}d =(\mathbf{1}_1 \otimes c)_{-2}(b \otimes d) \in C_2(V_1 \otimes V_2).
\end{array}
\end{align*}
As $C_2(V_1 \otimes V_2)$ is a vector space, it follows that $C_2(V_1) \otimes V_2+ V_1 \otimes C_2(V_2) \subseteq C_2(V_1 \otimes V_2)$, and so $C_2(V_1 \otimes V_2) = C_2(V_1) \otimes V_2+ V_1 \otimes C_2(V_1)$. We obtain an isomorphism of vector spaces between $R(V_1 \otimes V_2)$ and $R(V_1) \otimes R(V_2)$, and direct computations show that it is an isomorphism of algebras.
\end{proof}
It is mentioned in \cite{Frenkel-Huang-Lepowsky} that the vertex algebra $ V_1\oplus V_2$ together with the projections $ V_1\oplus V_2 \to V_i$ is the product in the category of vertex algebras. Thus the functor $R$ sends finite product to finite product. 

\begin{proof}[Proof of Theorem~\ref{XvandXv!}]
If  $V_1, V_2$ are  two vertex algebras in $\mathcal{V}$ (resp. in $\mathcal{V}_{f.g}$). Then $ V_1\oplus V_2$ and $ V_1\otimes V_2$ are also in $ \mathcal V $ (resp. in $\mathcal{V}_{f.g}$).   Applying the functor $\on{spec}$ and $\on{Max}$  to the isomorphisms in Lemma~\ref{lem:R(V)tensor}, we obtain the top four isomorphisms of the theorem.  

Set $V_1, V_2 \in \mathcal{VOA}_{f.g}$ and choose $\mathfrak{m}_{x_1} \in \on{Spec}(R(V_1)_0)$, $\mathfrak{m}_{x_2} \in \on{Spec}(R(V_2)_0)$. As $R(V_1 \oplus V_2) \cong R(V_1) \oplus R(V_2)$, we see that $R(V_1 \oplus V_2)_0 \cong R(V_1)_0 \oplus R(V_2)_0$. Therefore there are two natural actions of $R(V_1 \oplus V_2)$ on $\cc$: $R(V_1 \oplus V_2) \longrightarrow R(V_1)_0 \oplus R(V_2)_0 \stackrel{/\mathfrak{m}_{x_1}}{\longrightarrow} \cc_{x_1}$ and $R(V_1 \oplus V_2) \longrightarrow R(V_1)_0 \oplus R(V_2)_0 \stackrel{/\mathfrak{m}_{x_2}}{\longrightarrow} \cc_{x_2}$. We choose the action coming from $\mathfrak{m}_{x_1}$. 

Let $S \in H_{x_1}^*(R(V_1 \oplus V_2))$ be an $n$-fold exact sequence, i.e. $S=0 \longrightarrow \cc_{x_1} \longrightarrow S_{n-1} \xrightarrow{\partial_{n-1}} \dots \xrightarrow{\partial_1} S_0 \longrightarrow \cc_{x_1} \longrightarrow 0$, where the $S_j$ are $R(V_1 \oplus V_2)$-modules. As $R(V_1)$ and $R(V_2)$ are unital algebras, each module $S_i$ decomposes as $S_{i, 1} \oplus S_{i, 2}$ with $S_{i, j}$ an $R(V_j)$-module. We can see that each morphism $\partial_{i}$ can be decomposed as $\partial_{i, 1}+\partial_{i, 2}$ with $\partial_{i, j}: S_{i, j} \longrightarrow S_{i-1, j}$ a morphism of $R(V_j)$-modules. We can thus construct two $n$-fold exact sequences
\[
S^1=0 \longrightarrow \cc_{x_1} \longrightarrow S_{n-1, 1} \xrightarrow{\partial_{n-1, 1}} \dots \xrightarrow{\partial_{1, 1}} S_{0, 1} \longrightarrow \cc_{x_1} \longrightarrow 0,
\]
and
\[
S^2=0 \longrightarrow \{0\} \longrightarrow S_{n-1, 2} \xrightarrow{\partial_{n-1, 2}} \dots \xrightarrow{\partial_{1, 2}} S_{0, 2} \longrightarrow \{0\} \longrightarrow 0,
\]
with $S^1$ a sequence of $R(V_1)$-modules, $S^2$ a sequence of $R(V_2)$-modules, and such that $S=S^1 \oplus S^2$. But as $S^2$ starts and ends with $\{0\}$, it is equivalent to the zero sequence, and thus $S^2=0$. So $H_{x_1}^n(R(V_1 \oplus V_2)) \subseteq H_{x_1}^n(R(V_1))$. The reverse inclusion is true because an $n$-fold extension of $R(V_1)$-modules can be seen as a sequence of $R(V_1 \oplus V_2)$-modules where $R(V_2)$ acts trivially. As it is true for all $n \in \mathbb{N}$, we obtain that $H_{x_1}^{*}(R(V_1 \oplus V_2)) = H_{x_1}^{*}(R(V_1))$. If we start the same reasoning by looking at the action on $\cc_{x_2}$, we obtain $H_{x_2}^{*}(R(V_1 \oplus V_2)) = H_{x_2}^{*}(R(V_2))$. 

From what we have determined above, we now have $H_{x_1}^{*}(R(V_1 \oplus V_2))^{ab} = H_{x_1}^{*}(R(V_1))^{ab}$ and $H_{x_2}^{*}(R(V_1 \oplus V_2))^{ab} = H_{x_2}^{*}(R(V_2))^{ab}$. By taking the prime spectrum and maximal spectrum, we obtain the desired result.

Set $V_1, V_2 \in \mathcal{VOA}_{f.g}$ and define $\mathfrak{m}_{x_1} \in \on{Spec}(R(V_1)_0)$, $\mathfrak{m}_{x_2} \in \on{Spec}(R(V_2)_0)$. We know that $R(V_1 \otimes V_2) \cong R(V_1) \otimes R(V_2)$, so $R(V_1 \otimes V_2)$ acts naturally on $\cc$ through $R(V_1 \otimes V_2) \longrightarrow R(V_1)_0 \otimes R(V_2)_0 \longrightarrow \cc_{x_1} \otimes \cc_{x_2}=\cc_{(x_1,x_2)}$.
\delete{, where the last morphism is the quotient by $\mathfrak{m}_{x_1} \otimes R(V_2)_0 + R(V_1)_0 \otimes \mathfrak{m}_{x_2}$.  As the tensor product is over the complex field, it follows that the last tensor product is isomorphic to $\cc$ and we denote it by $\cc_{(x_1, x_2)}$. }

Using K\"unneth tensor product formula, one can directly determine that the tensor product of the minimal projective resolutions of $\cc_{x_1}$ as $R(V_1)$-module and of $\cc_{x_2}$ as $R(V_2)$-module is a minimal projective resolution of $\cc_{(x_1, x_2)}$ as an $R(V_1 \otimes V_2)$-modules. We now look at the wedge product (cf. \cite[\textrm{VIII}.4]{MacLane}): 
\begin{align*}
\psi: H_{x_1}^*(R(V_1)) \otimes H_{x_2}^*(R(V_2)) \longrightarrow H_{(x_1, x_2)}^*(R(V_1\otimes V_2)).
\end{align*}
Assuming $R(V_1)$ and $R(V_2)$ are local algebras, there exists a minimal free resolution $X_1$ (respectively $X_2$) of $\cc_{x_1}$ as $R(V_1)$-module (respectively $\cc_{x_2}$ as $R(V_2)$-module) (cf. \cite[Chapter 19]{Eisenbud}). We then use the facts that $X_1 \otimes X_2$ is a minimal projective resolution of $\cc_{(x_1, x_2)}$ as $R(V_1 \otimes V_2)$-module, and that each space $X_j^n$, $j=1,2$, has a basis to compute $\psi$ and see that it is an isomorphism of algebras. 

\delete{Set $\mathcal{I}_{x_1}$ and $\mathcal{I}_{x_2}$ the sets of ideals defining the inverse systems for $H_{x_1}^*(R(V_1))$ and $H_{x_2}^*(R(V_2))$. If $I_1 \in \mathcal{I}_{x_1}$ and $I_2 \in \mathcal{I}_{x_2}$, then the ideal $(I_1 , I_2)=I_1 \otimes H_{x_2}^*(R(V_2))+ H_{x_1}^*(R(V_1))\otimes I_2$ verifies: 
\begin{align*}
\left(H_{x_1}^*(R(V_1)) \otimes H_{x_2}^*(R(V_2)) \right)/(I_1 ,I_2) \cong H_{x_1}^*(R(V_1))/I_1 \otimes H_{x_2}^*(R(V_2))/I_2,
\end{align*}
which is commutative. By denoting $\mathcal{I}_{(x_1, x_2)}$ the set of ideals defining the inverse systems for $H_{(x_1, x_2)}^*(R(V_1 \otimes V_2))$, we see that $(I_1, I_2) \in \mathcal{I}_{(x_1, x_2)}$. Let us set
\begin{align*}
\begin{array}{l}
L=[H_{x_1}^*(R(V_1)) \otimes H_{x_2}^*(R(V_2)), H_{x_1}^*(R(V_1)) \otimes H_{x_2}^*(R(V_2))] \\[5pt]
\quad =[H_{x_1}^*(R(V_1)), H_{x_1}^*(R(V_1))] \otimes H_{x_2}^*(R(V_2))+H_{x_1}^*(R(V_1)) \otimes [H_{x_2}^*(R(V_2)), H_{x_2}^*(R(V_2))].
\end{array}
\end{align*}
It is clear that any ideal in $J \in \mathcal{I}_{(x_1, x_2)}$ contains $L$. Hence for all $J \in \mathcal{I}_{(x_1, x_2)}$, we have $L  \succcurlyeq J$, and so any element $a_J$ in the inverse limit of $H_{(x_1, x_2)}^*(R(V_1 \otimes V_2))$ is the image of some $a_L$. As $L \in \mathcal{I}_{x_1} \times \mathcal{I}_{x_2}$, we obtain an isomorphism:
\begin{align*}
H_{(x_1, x_2)}^*(R(V_1 \otimes V_2))^{ab}  \longrightarrow  \underset{(I_1, I_2) \in \mathcal{I}_{x_1} \times \mathcal{I}_{x_2}}{\varprojlim} \left( H_{x_1}^*(R(V_1))/I_1 \otimes H_{x_2}^*(R(V_2))/I_2 \right).
\end{align*}
If $(I_1 , I_2) \subseteq (I_1' , I_2')$, then $I_1 \subseteq I_1'$ and $I_2 \subseteq I_2'$. The reciprocal is also true. So $(I_1 , I_2) \succcurlyeq (I_1' , I_2')$ is equivalent to $I_1  \succcurlyeq I_1'$ and $ I_2 \succcurlyeq  I_2'$. It follows that the transition maps are compatible and so we can define an isomorphism:
\begin{align*}
\begin{array}{ccc}
 \underset{(I_1 , I_2) \in \mathcal{I}_{x_1} \times \mathcal{I}_{x_2}}{\varprojlim} \left( H_{x_1}^*(R(V_1))/I_1 \otimes H_{x_2}^*(R(V_2))/I_2 \right) & \longrightarrow & H_{x_1}^{*}(R(V_1))^{ab} \otimes H_{x_2}^{*}(R(V_2))^{ab} \\
 (a_{I_1} \otimes b_{I_2})_{(I_1 , I_2) \in \mathcal{I}_{x_1} \times \mathcal{I}_{x_2}} & \longmapsto &  (a_{I_1})_{I_1 \in \mathcal{I}_{x_1}} \otimes  (b_{I_2})_{I_2 \in \mathcal{I}_{x_2}}.
\end{array}
\end{align*}}
It is straightforward to verify using the universal property that,  for any two associative  $ \cc$-algebras  $A_1$ and $A_2$, there is a natural isomorphism of $\cc$-algebras $ A_1^{ab}\otimes A_2^{ab}\stackrel{\sim}{\to}(A_1 \otimes A_2)^{ab}$. Thus we have 
\begin{align*}
H_{(x_1, x_2)}^*(R(V_1 \otimes V_2))^{ab}  \stackrel{\cong}{\longrightarrow} H_{x_1}^{*}(R(V_1))^{ab} \otimes H_{x_2}^{*}(R(V_2))^{ab}.
\end{align*}
The prime spectrum and maximal spectrum transform the tensor product into a direct product, and the desired result follows.
\end{proof}

Let $V$ be an $\mathbb{N}$-graded vertex operator algebra with $V_0=\cc \mathbf{1}$ and $R(V)$ finitely generated. Arakawa showed (\cite{Arakawa1}) that $V$ is finitely strongly generated and thus, that $\on{dim}_\cc \ X_V=0$ if and only if $V$ is $C_2$-cofinite. As $X_V$ is conical, if $X_V$ contains a point different from $0$ it must then contain a line. Thus the dimension of $X_V$ is $0$ if and only if $X_V=\{0\}$. We would like to obtain a similar description of $X_{V, x}^{!}$. 

Let $V$ be an $\mathbb{N}$-graded vertex operator algebra with finitely generated $R(V)$. Assume that $R(V)$ has finite global dimension and choose $\mathfrak{m}_x \in \on{Max}(R(V)_0)$. Then $\cc_x$ has finite projective dimension in $R(V)\Mod$, and so $H_x^{*}(R(V))= \bigoplus_{n=0}^kH_x^{n}(R(V))$ for some $k \in \mathbb{N}$. It follows that all elements in $H_x^{*}(R(V))_{\geq 1}$ are nilpotent.\delete{and thus $H_x^{*}(R(V))_{\geq 1} \subseteq \on{Rad}(H_x^{*}(R(V)))$. However} 
Since $H_x^{*}(R(V))/H_x^{*}(R(V))_{\geq 1}  \cong H_x^{0}(R(V))= \cc$, \delete{the ideal $H_x^{*}(R(V))_{\geq 1}$ is maximal in $H_x^{*}(R(V))$. It follows that $\on{Rad}(H_x^{*}(R(V)))$ is a maximal ideal, and it is the unique maximal ideal of $H_x^{*}(R(V))$.} thus  $H_x^{*}(R(V))^{ab}$ is also local and $X_{V, x}^{!}=\{0\}$. 
\delete{, and so for any $I \in \mathcal{I}_x$, we have $\on{Max}(H_x^{*}(R(V))/I)=\{0\}$, and finally } 

\delete{We summarize the behaviour of $X_V$ and $X_{V, x}^{!}$ below. If $R(V)$ is finitely generated, then:
\[
  \begin{tikzpicture}[scale=0.9,  transform shape]
  \tikzset{>=stealth}
\node (1) at ( 0,0){$R(V)$ has finite dimension};
\node (2) at ( 6,0){$X_V=\{0\}$};
\node (3) at ( 0,-3){$R(V)$ has finite global dimension};
\node (4) at ( 6,-3) {$X_{V, x}^{!}=\{0\}$};
\node (5) at ( 9.2,-3) {for all $\mathfrak{m}_x \in \on{Max}(R(V)_0)$};
\node (6) at ( 1.5,-1.5){dual statements};
\node (7) at ( 7.5,-1.5){dual statements};

\draw [decoration={markings,mark=at position 1 with
    {\arrow[scale=1.2,>=stealth]{>}}},postaction={decorate}] (1)  --  (3);
    \draw [decoration={markings,mark=at position 1 with
    {\arrow[scale=1.2,>=stealth]{>}}},postaction={decorate}] (3)  --  (1);

\draw [decoration={markings,mark=at position 1 with
    {\arrow[scale=1.2,>=stealth]{>}}},postaction={decorate}] (2)  --  (4);
    \draw [decoration={markings,mark=at position 1 with
    {\arrow[scale=1.2,>=stealth]{>}}},postaction={decorate}] (4)  --  (2);

\draw[implies-implies,double equal sign distance] (1) -- (2);
\draw[-implies,double equal sign distance] (3) -- (4);

\end{tikzpicture}
\]}

\begin{remarks}
(1) Notice that on the second line we only have an implication and not an equivalence. It is not yet clear if $X_{V, x}^{!}=\{0\}$ implies that $R(V)$ has finite global dimension. 

(2) If $V$ is a vertex operator algebra of CFT type,  then the $C_2$-algebra is a positively graded, augmented algebra with $R(V)_0=\cc$. Thus the affine point $\{0\}$ corresponds to the kernel of the augmentation map $R(V) \longrightarrow \cc$. Likewise, the cohomology ring $H^{*}(R(V))$ is here also positively graded and augmented. \delete{so the point $\{0\}$ corresponds to the kernel of the augmentation map $H^{*}(R(V)) \longrightarrow \cc_x$. }
\end{remarks}

The next proposition illustrates that the behaviours of $X_V$ and $X_{V, x}^{!}$ are not free:

\begin{proposition}
Let $A$ be a finitely generated, commutative algebra over $\cc$ with augmentation $\epsilon: A \longrightarrow \cc$, such that $H^*(A)$ is finitely generated by homogeneous elements. Furthermore, let us assume that $\on{dim}_\cc \on{Ker} \epsilon / (\on{Ker} \epsilon)^2 >0$. Let $\{\alpha_1, \dots, \alpha_n\}$ be an ordered set of the generators of $H^*(A)$. We also assume that the product in $H^*(A)$ is such that any element of $H^*(A)$ can be written as a linear combination of ordered monomials $\alpha_1^{i_1} \dots \alpha_n^{i_n}$. Finally, we assume that $H^{*}(A)/\on{Rad}(H^{*}(A))$ is a commutative algebra. Then:
\begin{itemize}\setlength{\itemsep}{5pt}
\item if $\on{Max}(A)=\{pt\}$ then $\on{Max}(H^{*}(A)^{ab}) \neq \{pt\}$,
\item if $\on{Max}(H^{*}(A)^{ab})=\{pt\}$ then $\on{Max}(A)\neq \{pt\}$.
\end{itemize}
\end{proposition}

\begin{proof}
Let $A$ be as in the proposition and assume that $\on{Max}(A)=\{pt\}=\on{Max}(H^{*}(A)^{ab})$. Note that every finitely generated commutative local $\cc$-algebra $A$ is finite dimensional. Indeed, since  $A$ is commutative and finitely generated over $\cc$, $ A$  is a Jacobson ring and  its Jacobson radical is equal to its nilradical. 
As $A$ is local, its unique maximal ideal is equal to its nilradical, and so each generator is nilpotent. Because $A$ is finitely generated, it is then also finite dimensional. 

The locality of $A$ implies the existence of a unique simple $A$-module. Thus the Ext quiver $Q_{A}$ of $A$ has a unique vertex. Furthermore, the number of loops on this vertex is $\on{dim}_\cc  H^1(A)$, which is not zero because it is equal to the minimal number of generators of $A$, i.e. $\on{dim}_\cc \on{Ker} \epsilon / (\on{Ker} \epsilon)^2>0$. 

However, it was proved by K. Igusa (\cite{Igusa}) that if a finite dimensional algebra has finite global dimension, then its Ext quiver has no loops. As we just saw $Q_{A}$ has loops, it implies that $\text{gl.dim }A=+\infty$. But we also know that the global dimension of a finite dimensional algebra is the maximum of the projective dimensions of its simple modules (\cite[A.4 Theorem 4.8]{Assem-Simson-Skowronski}), and $\cc$ is the only simple $A$-module. It follows that $\text{proj.dim }\cc=+\infty$, and $H^*(A)$ is infinite dimensional.
Under the assumption, $H^{*}(A)^{ab}$ is a finitely generated  commutative $\cc$-algebra. Thus the assumption that  $H^{*}(A)^{ab}$ is a local algebra would imply that $ H^{*}(A)^{ab}$ is finite dimensional.  We want to show that, under  the additional assumption that $ H^*(A)/\on{Rad}(H^*(A)) $ commutative,  the algebra $H^{*}(A)$ is finite dimensional. This follows from the following lemma, and therefore contradicts the fact that $ H^*(A)$ is infinite dimensional.
\end{proof}
\begin{lemma} Let $ B=\oplus_{i\geq 0}B_i$ be a graded associative (not necessarily commutative) $\bb C$-algebra with a set of homogeneous generators $c_1, \dots, c_n$ of positive degrees such that $ B=\cc \on{-Span}\{ c_1^{l_1}c_2^{l_2}\cdots c_n^{l_n}\;|\; (l_1, \dots, l_n)\in \bb N^{n}\}$ (with a fixed order). If  $B/\langle [B, B]\rangle $ is local and $ B/\on{Rad}(B)$ is commutative, then $ B$ is finite dimensional. 
\end{lemma}
\begin{proof}
Under the assumption, for any maximal (two sided) ideal $ \frak m$ of $ B$, $B/\frak m$ is a quotient of $ B/\on{Rad}(B)$ and thus a commutative algebra. Hence  $ \frak m\supseteq  \langle [B, B]\rangle $ and therefore $ \on{Rad}(B)\supseteq \langle [B, B]\rangle$.  Thus $B/\on{Rad}(B)$ is a quotient of $B^{ab}$. The locality of $ B^{ab}$ implies that $ B/\on{Rad}(B)=\cc$ and $\on{Rad}(B)=B_+=\oplus_{i>0}B_i$.  Thus for every homogeneous  $x \in B_{n}$ with $ n > 0$, $ 1-x$ is invertible in $ B$.  Let $b=\sum_{i=0}^s b_i\in B$ such that $ b_i \in B_i$ and $ (1-x)b=1$. 
Using the graded structure comparing the homogeneous components we get $ b_{j n}=x^j$ and $ b_i=0$ for other $ i\not\in n\bb N$. Therefore $ x^j=0$ when $j n>s$. Thus all $x\in B_n$ are nilpotent for all $ n>0$.  Finally, we use the assumption $ B=\cc\on{-Span}\{ c_1^{l_1}c_2^{l_2}\cdots c_n^{l_n}\;|\; (l_1, \dots, l_n)\in \bb N^{n}\}$ to conclude that $ B$ is finite dimensional. 
\end{proof} 

\begin{remark}
\delete{(1) The graded algebra structure is crucial. For example, let $ B=U(\frak g)$ be the universal enveloping algebra of a finite dimensional semisimple Lie algebra $\frak g$. The property $ [\frak g, \frak g]=\frak g$ implies that $U(\frak g)^{ab}=\cc$ and thus local. It also satisfies the $\cc\on{-span}$ condition by the PBW theorem. But $ U(\frak g)$ has no nonzero nilpotent element.  Of course,  in this case, $B/\on{Rad}(B)=U(\frak g)$ is not commutative (by noting that the adjoint representation on $\frak g$ induces a faithful $U(\frak g)$-module structure on $ \on{Sym}(\frak g)$ by the PBW theorem). }
 
The hypothesis on the generators of $H^*(A)$ means that the Yoneda product behaves relatively well. Indeed, we might not have a clear commutation rule, but we can at least order the generators. These conditions will be proved to hold when $ A$ is a complete intersection. This will discussed in \cite{Caradot-Jiang-Lin_2}.
\end{remark}

\begin{remarks}\label{rem:3.4} 
(1)  We see that in the context of a vertex operator algebra $V$, we can apply the above proposition to $R(V)$. If $R(V)$ abides by the hypothesis of the proposition, then we cannot have $X_V=\{0\}$ and $X_{V, x}^{!}=\{0\}$. 

(2)  The classical examples of non-$C_2$-cofinite vertex operator algebras are the Heisenberg vertex operator algebra $V_{\hat{\mathfrak{h}}}(k,0)$, the universal affine vertex operator algebra $V_{\hat{\mathfrak{g}}}(k,0)$, and the universal Virasoro vertex operator algebra $V_{\text{Vir}}(c,0)$. For these examples, the $C_2$-algebra $R(V)$ is a polynomial algebra $\cc[x_1,\dots ,x_n]$ and $\on{Spec}(R(V)_0)=\{0\}$. Thus $H_{\{0\}}^{*}(R(V))=\wedge[x_1, \dots, x_n]$ and so $X_V=\bb A^n$ and $X_{V, 0}^{!}=\{0\}$. 

(3)  It is possible to have both $X_V \neq \{0\}$ and $X_{V, x}^{!} \neq \{0\}$. Indeed, take $V=V_{\hat{\mathfrak{g}}}(k,0) \otimes L_{\hat{\mathfrak{g}}}(l,0)$ with $l \in \mathbb{N}$. Then using Theorem~\ref{XvandXv!}:
\begin{itemize}
\item $X_V=X_{V_{\hat{\mathfrak{g}}}(k,0)} \times X_{L_{\hat{\mathfrak{g}}}(l,0)} \cong \mathfrak{g}^*$,
\item $X_{V, (0,0)}^{!} \cong X_{V_{\hat{\mathfrak{g}}}(k,0), 0}^{!} \times X_{L_{\hat{\mathfrak{g}}}(l,0), 0}^{!} \cong X_{L_{\hat{\mathfrak{g}}}(l,0), 0}^{!}$ which is not zero if $l \geq 2$ (cf. Theorem~\ref{dimX_V!LieAlgebra}). Here we cannot apply Theorem~\ref{XvandXv!} directly because $R(V_{\hat{\mathfrak{g}}}(k,0))$ is not local, but $\cc_{\{0\}}$ has a minimal free (finite) resolution as an $R(V_{\hat{\mathfrak{g}}}(k,0))$-module, so the rest of the theorem works. This implies that $X_{V, (0,0)}^{!}$ is not zero.
\end{itemize}
\end{remarks}

\section{Commutative quotients of Ext algebras of finitely generated commutative algebras}\label{Sectiondimcomology}
Computations of cohomological spaces tend to be difficult, and so the cohomological variety, if not trivial, is not easy to determine explicitly. In this section, for a finitely generated commutative algebra $R$ and a maximal ideal $\frak m$ with $R/\frak m=\bf k$, our goal is to determine an affine subscheme of $ \on{Spec}(\on{Ext}^*_{R}(\bfk, \bfk)^{ab})$.

\subsection{Tate's construction of a free  resolution}\label{SectionTate}
The starting point of our description is based on a construction of J. Tate (\cite{Tate}). We give below a brief presentation of his results.

Let $S$ be a commutative Noetherian ring with unit and containing a subfield of characteristic $0$.  
An $ \bb N$-graded {\em commutative differential graded} $S$-algebra  is a pair $X=(X_*, \partial)$, a graded associative $S$-algebra $X=\oplus_{n\geq 0}X_n$ equipped with an $S$-linear operator $\partial :X \longrightarrow X$ of degree $-1$ such that:
\vspace{-\topsep}
\begin{itemize}
\item $X=\bigoplus_{i \geq 0}X_i$ is positively graded by $S$-submodules, and $X_i X_j \subseteq X_{i+j}$,
\item $X$ has a unit $1$, $X_0=S1$, and the $X_i$'s are finitely generated $S$-modules,
\item $X$ is strictly skew-commutative, i.e \begin{align*}
\begin{array}{cr}
xy=(-1)^{ij}y x & \text{ for }x \in X_i, y \in X_j,
\end{array}
\end{align*}
and 
\begin{align*}
\begin{array}{cr}
x^2=0 & \text{ if }x \in X_i \text{ with } i \text{ odd},
\end{array}
\end{align*}
\item The map $\partial$ is a skew-derivation of degree $-1$, i.e. $\partial(X_i) \subseteq X_{i-1}$ for all $i$, $\partial^2=0$, and 
\begin{align*}
\begin{array}{cr}
\partial(xy)=(\partial x)y+(-1)^{i}x(\partial y) & \text{ for }x \in X_i, y \in X_j.
\end{array}
\end{align*}
\end{itemize}
By writing $\partial_{i}$ for the restriction of $\partial$ to $X_{i}$, an $S$-algebra $X$ can therefore be seen as a complex:
\begin{align*}
\cdots \longrightarrow X_n \stackrel{\partial_n}{\longrightarrow} X_{n-1} \longrightarrow \cdots \cdots \longrightarrow X_1 \stackrel{\partial_1}{\longrightarrow} X_0 \longrightarrow 0.
\end{align*}
If $X$ is acyclic (i.e., $H_n(X)=0$ for all $n>0$) and each $X_i$ is a free $S$-module, then it gives a free resolution of the $S$-module $H_0(X)=S/\partial_1(X_1)$:
\begin{align*}
\cdots \longrightarrow X_n \stackrel{\partial_n}{\longrightarrow} X_{n-1} \longrightarrow \cdots \cdots \longrightarrow X_1 \stackrel{\partial_1}{\longrightarrow} S \stackrel{\epsilon}{\longrightarrow} S/\partial_1(X_1) \longrightarrow 0.
\end{align*}
If $X=(X_*, \partial)$ is not necessarily  acyclic, but each $X_i$ is still a free $S$-module, the following is a process  to construct an another commutative differential graded $S$-algebra $Y=(Y_*, \partial)$, inductively by adding homogeneous variables to $ X$ to form a commutative differential algebra extension $Y\supseteq X$ so that the complex $(Y_*, \partial)$ becomes acyclic.  J. Tate described it as ``killing the homology" of $X$.

\delete{
Let $ Z_j(X)=\ker(\partial_j)$ the $S$-module of $ j$-cycles and $ B_j(X)=\on{im}(\partial_{j+1})$ be the $S$-module $j$-boundaries. 
Assume $H_j(X)=Z_j(X)/B_j(X)\neq 0$ (which is an $S$-module) and $H_j(X)=0$ for all $ i<j$.
 Then $H_j(X)$ is a graded $ S$-module concentrated in degree $j$. Consider the shift $ H_j(X)[1]$ which is concentrated in degree $j+1$. 
 Let $ \on{Sym}_S(H_j(X)[1])$ be the differential graded symmetric $S$-algebra. Here for any $ \bb Z$-grade $ S$-module the tensor algebra $ T_S(V)$ naturally graded by $\deg(v\otimes v')=\deg(v)+\deg(v')$ for homogeneous $ v, v'\in V$.  Let $\langle [V, V]\rangle $ be the ideal of $ T_S(V)$  generated by all $ [v, v']=v\otimes v' -(-1)^{\deg(v)\deg(v')}$ with $ v, v'\in V$ homogeneous.  Thus $ \on{Sym}_S(V)=T_S(V)/\langle [V, V]\rangle$ is a commutative differential graded $S$-algebra which is also free with zero differential $ \partial =0$. 
 $ \on{Sym}_S(V)$ is free in the sense 
    that for any commutative differential algebra $ (A,\partial)$ any homogeneous $S$-module homomorphism $ \phi: V\to A$ such that $\partial \phi(v)=0$ for all $ v\in V$, i.e., $ \phi: v\to Z(A)$ is a homogeneous $S$-module homomorphism, then there is unique $S$-algebra homomorphism $\on{sym}_S(V)\to A$. }

For  $j >0$  a positive integer, we denote by $Z_{j-1}(X)=\on{Ker}(\partial_{j-1})$ and $B_{j-1}(X)=\on{Im}(\partial_{j})$ the groups of cycles and boundaries of the complex $X$. Let $t \in Z_{j-1}(X)\setminus B_{j-1}(X)$ be a cycle of degree $j-1$. The inductive step is to construct an extension $X\to Y$ of commutative differential graded $S$-algebras  such that:
\vspace{-\topsep}
\begin{itemize}
\item $Y_i= X_i$ for $i<j$, and,
\item $Y_j=X_j\oplus ST$ with $ T $ of degree $j$.
\item $B_{j-1}(Y)= B_{j-1}(X) + St$ with $\partial_j (T)=t$.
\end{itemize}
 The inclusion $X\subseteq Y$ as differential graded algebras induces a homomorphism of graded algebras $ H_*(X)\to H_*(Y)$ such that  the element $t$ is sent to zero. The construction differs depending on the parity of $j$:

- $j$ odd. Let $XT$ be the free $X$-module with one basis element $T$, and set $Y=X\langle T\rangle=X \oplus XT$. We grade $Y$ by giving $T$ the degree $j$; that is, we write $Y_i= X_i+X_{i-j}T$ for all $i \geq 0$. This defines $Y$ as a graded $S$-module. The graded $S$-algebra structure  is defined by $T^2=0$ and  the super commuting relation $ Tx=(-1)^{\deg(x) \deg(T)} xT$. Note that $T$ does not commute with odd degree elements in $X$. Now  there is a unique way to extend $\partial$ to $Y$ such that $\partial_j T=t$ by product rule $ \partial (xT)=\partial (x)T+(-1)^{\deg(x)}x t. $  Thus $ (Y, \partial)$ is commutative differential graded algebra and $ X\to Y$ is a homomorphism of differential graded algebras. 

- $j$ even. Set $Y=X\langle T\rangle =X[T]$, the ring of polynomials with one commuting variable and coefficients in $X$. By setting the degree of $T$ as $j$ and the derivation as $\partial_j T=t$ and using the product rule $ \partial (xT^n)=\partial (x)T^n+(-1)^{\deg(x)}n x t T^{n-1}$, we obtain the desired commutative differential graded algebra $S$-algebra $Y$. 

We remark that in general, $X\langle T\rangle $ should be defined by using the divided powers of $T$ rather than polynomial algebra. Since we will assume that $S$ contains $\bb Q$, the field of rational numbers, the above construction is equivalent to Tate's construction.

In both cases, $j$ even and $j$ odd,   we denote the commutative differential graded $S$-algebra $Y=X\langle T \rangle$ with $\deg(T)=j$ and   $ \partial T=t \in Z_{j-1}(X)$. 

In general, for any subset $ C_{j-1}\subseteq Z_{j-1}(X)$ such that $ B_{j-1}(X)+SC_{j-1}=Z_{j-1}(X)$, take a free $S$-module  $F_{j}$ with basis $C_{j-1}$ with $ S$-module homomorphism $ \phi: F_{j}\to Z_{j-1}(X)$.  We define $ S\langle F_{j}\rangle =\wedge_{S}(F_{j})$ if $ j$ is odd and $S\langle F_{j}\rangle=\on{Sym}_S(F_j)$ as usual for $j$ even as graded algebra with $ F_j$ in degree $j$. Let $ F_j[j]$ be the graded $ S$-module with $ F_j$ concentrated in degree $ j$. Then $S\langle F_{j}\rangle=\on{Sym}_S(F_j[j])$ is the symmetric algebra in the symmetric tensor category of graded $S$-modules with tensor product $ \otimes_S$ and braiding $ (x\otimes y)\to (-1)^{\deg(x)\deg(y)}(y \otimes x)$. We write $ X\langle F_j\rangle=X\otimes_S S\langle F_j\rangle $. Thus the algebra structure in $X\otimes_S S\langle F_j\rangle$ is defined by $ (x\otimes_S f)(x'\otimes_S f') =(-1)^{\deg(x')\deg(f)}xx'\otimes_S ff'$ for homogeneous elements $ x, x'\in X$ and $ f, f'\in S\langle F_j\rangle $. The graded algebra $X\langle F_j\rangle$ has a differential uniquely extending the one on $X$ using $ \partial (x\otimes v)=\partial(x)\otimes v+(-1)^{\deg(x)}x\phi(v) \otimes 1$ for all $ v\in F_j$.  Then $Y=X\langle F_j\rangle $ has the above properties (i)-(iii), which are modified as
\begin{itemize}
\item[(i)] $  X\subseteq Y$ is a differential subalgebra, and $Y_i= X_i$ for $i<j$,
\item[(ii)] $Y_j=X_j\oplus F_j $,
\item[(iii)]  $B_{j-1}(Y)=B_{j-1}(X)+SC_{j-1}\subseteq Z_{j-1}(Y)=Z_{j-1}(X)$.
\end{itemize}
Note that if all summands $X_i$ are free $S$-modules, so are all the summands $Y_i$. 

We remark that $X$ is a commutative differential graded algebra and $\phi:  F_{j}[j-1]\to X$ is a chain map. Then $ X\langle F_j\rangle=X\otimes_{S}\on{Sym}_{S}(F_j[j])=\on{Sym}_{X}(X\otimes S F_j[j])$ is the Koszul complex of the chain map $ X\otimes_S  F_j[j-1]\to X$ in the dg setting.

\subsection{Tate resolutions corresponding to ideals} Let $I$ be an ideal of  $S$.  Taking $X^0=S$  as a commutative differential graded algebra  concentrated in degree $0$. Note that $H_0(X^0)=Z_0(X^0)=S$.  Take $ C_0$ as a minimal generating subset of the ideal $ I \subseteq Z_0(X)=S$, we get a commutative differential graded algebra $X^1=S\langle F_1\rangle$ from the above construction. $X^1$ is exactly the Koszul complex of the map $ F_1\to S$ in the classical sense. 
 We inductively define an $\bb N$-graded commutative differential algebra $X^j$ such that each $ X_i^{j}$ is a free $ S$-module with a basis $ \cal B(X^j_i)$ and satisfies \begin{itemize}
\item $ X^i\subseteq X^j$ is  a differential graded subalgebra for each $ i\leq j$;
\item $X^j_i=X^i_i$ and $ \cal B(X^i_i)=\cal B(X^j_i)$ for all $ 0<i\leq j$;
\item $H_i(X^{j})=0$  for all $ i< j$.
\end{itemize}
Assume that $ X^j$ has been constructed for $j>0$. If $ H_j(X^j)\neq 0$, then choose a subset $ C_j\subseteq Z_j(X^j)$ such that their image in $ H_j(X^j)$ generate  $H_j(X^j)$  as $S$-module. Let $ F_{j+1}$ be the free $S$-module with basis $C_j$. Then $F_{j+1}[j+1]$ is a graded $S$-module with $ F_{j+1}$  in degree $ j+1$.  Thus define $ S\langle F_{j+1}\rangle=\on{Sym}_S(F_{j+1}[j+1])$ to be the symmetric algebra in the symmetric tensor category of differential complexes of  $ S$-modules with tensor product $ \otimes _S$. 
Note that $ \on{Sym}_S(F_{j+1}[j+1])$ is commutative  differential graded algebra with differential $\partial=0$ and  has the universal property that for any commutative  graded  $S$-algebra $ A$ and any set map $ \rho: C_j\to Z_j(A)$, there is a unique differential graded algebra  homomorphism $ \on{Sym}_S(F_{j+1}[j+1])\to A$ extending the map $ \rho$ (see \cite[1.1]{Iyengar}).
We now define $X^{j+1}=X^j\otimes_S \on{Sym}_S(F_{j+1}[j+1])$ as graded algebra with multiplication defined by $ (x\otimes_S f)(x'\otimes_S f') =(-1)^{\deg(x')\deg(f)}xx'\otimes_S ff'$ and the differential $ \partial $ is defined uniquely extending the one on $X^j$ by $ \partial (x\otimes v)=\partial(x)\otimes v+(-1)^{\deg(x)}x\phi(v) \otimes 1$  with $ v\in F_{j+1}$.  

Then $ X^{j}=X^{j}\otimes_S S\subseteq X^{j+1}$ is a differential graded subalgebra. Define $ \cal B(X^{j}_i)=\cal B(X^{j+1}_i)$ for all $i\leq j$ and $ \cal B(X^{j+1}_{j+1})=\cal B(X^{j}_{j+1})\cup C_j$ and $ \cal B(X^{j}_{l})\subseteq \cal B(X^{j+1}_{l})$ for all $ l\geq j+1$ by noting that 
\[ X^{j+1}_{l}=\bigoplus_{s\geq 0} X_{l-s}^{j}\otimes_{S}\on{Sym}^s_S(F_{j+1}[j+1]) .\]

Finally we get a commutative graded differential algebra $ X=\bigcup_{j}X^{j}$, which is a commutative differential graded algebra over $S$  with $H_i(X)=0$ for all $ i>0$, $ H_0(X)=S/I$, and all summands $ X_i$ are free modules with a basis $\cal B(X_i)$.  

The tower of commutative differential graded algebras $X^j\subseteq X^{j+1} \subseteq \cdots \subseteq X$ is called a Postnikov tower. See \cite{Buchweitz-Roberts}. 

It also follows from the construction that the basis is multiplicative in the following sense: for any $ b\in \cal B(X_i)$ and $b'\in \cal B(X_j)$, there is unique basis element $b''\in \cal B(X_{i+j})$ such that $bb'=\pm b''$. This can be achieved by choosing a total order on the set $C_j$ which gives the ordered basis of $ F_{j+1}$, and thus an ordered  basis $\cal B(X^{j+1})$ whose elements are of the form $ bc_{n_1}^{a_{1}}\cdots c_{n_d}^{a_{d}}$ (with $\{ c_{n_1}<c_{n_2}<\cdots<c_{n_d}\}$ being finite ordered subset of $ C_j$ and  $b \in \cal B(X^{j}))$. 

  The following result can then be obtained (\cite[Theorem 1]{Tate}). 

\begin{theorem} \label{thm:tate-resolution-ideals}
Let $ I$ be any ideal in $S$. Then there exists a commutative differential graded $S$-algebra $X=(X_*, \partial)$ where each summand $X_i$ is $S$-free with basis $ \cal B(X_i)$ such that $b\in \cal B(X_i)$ and $ b'\in\cal B(X_j)$ implies either $ bb'$ or $-bb'$ is in $ \cal B(X_{i+j})$. Furthermore, $H_0(X)=S/I$ and $H_i(X)=0$ for $i>0$. 
\end{theorem}
We call the commutative differential graded algebra $X$ in this theorem the Tate resolution of the homomorphism $ S\to S/I$. We remark that the resolution constructed in this theorem is by no means unique. It depends on the choice of the free module $ F_{j+1}$ (rather the set $C_j$) at each step. In case $S$ is Noetherian, then $ I$ is finitely generated and $Z_{j}(X^j)$ are all finitely generated $S$-modules. Thus one can choose a minimal number of generators $C_j$ so that the number of variables at each step is as small as possible.

\subsection{Compatible Tate resolutions}
Let $\pi: \tilde S\to S$ be a  homomorphism of commutative rings. Let $ I\subseteq S$ and $\tilde I\subseteq \tilde S$  be ideals such that $ \pi(\tilde I)\subseteq I$. Then $ \pi$ induces a ring homomorphism $ \tilde S/\tilde I\to S/I$. We simply regard all $ S$-modules as $\tilde S$-modules via the restriction functor $ \pi^*: S\Mod\to \tilde S\Mod$. 

\begin{theorem} \label{thm:compatible_tate_resolution} There are compatible Tate resolutions $ \tilde X $ of $ \tilde S\to \tilde S/\tilde I$ and $  X $ of $  S\to  S/ I$ with bases $ \cal B(\tilde X)$ and $\cal B(\tilde X)$ in the sense that there is a  differential graded algebra homomorphism $ \tilde X\to X$ which extends the exact sequence 
\[\xymatrix{ \cdots\ar[r]&\tilde X_2\ar[r]\ar[d]_{\pi_2}&\tilde X_1\ar[d]_{\pi_1}\ar[r] &\tilde  S\ar[r]\ar[d]_{\pi}& \tilde S/\tilde I\ar[d]\\
 \cdots\ar[r]& X_2\ar[r]& X_1\ar[r] &  S\ar[r]&  S/I
}\]
and $\pi_n (\cal B(\tilde X_n)) \subseteq \cal B(X_n)\cup \{0\}$. In particular, $\pi_n$ factors through $S\otimes_{\tilde S}\tilde X_n\to X_n$ with splitting cokernel as $ S$-module. 
\end{theorem}

\begin{proof}
We follow the construction of the above theorem and assume that $\tilde X^j$ and $ X^j$ have been constructed with a homomorphism $\pi_j: \tilde X^j\to X^j$ of commutative differential algebras. Then $ \pi_j(Z_j(\tilde X^j))\subseteq Z_j(X^j)$ and  $ \pi_j(B_j(\tilde X^j))\subseteq B_j(X^j)$, inducing a homomorphism $ H_j(\pi): H_j(\tilde X^j)\to H_j(X^j)$ since $ \pi$ is a chain map. By choosing a set of generators of the $S$-module $H_j(X^j)$ extending a set of generators of the image  $H_j(\pi)( H_j(\tilde X^j))$, we can choose $C_{j}\subseteq Z_j(X^j)$  and $\tilde C_{j}\subseteq Z_j(\tilde X^j)$ such that  $Z_j( X^j)=B_j( X^j)+ S  C_{j}$,  $Z_j(\tilde X^j)=B_j(\tilde X^j)+\tilde S \tilde C_{j}$,  and $ \pi_j(\tilde C_{j})\subseteq C_{j}\cup \{0\}$. Thus 
there is an $\tilde S$-module  homomorphism $ \tilde F_{j+1}\to F_{j+1}$ of $ \tilde S$-modules sending basis elements to basis elements or zero and making the following diagram commute:
\[\xymatrix{\tilde F_{j+1}\ar[r]^{\tilde \rho}\ar[d]_{}& Z_j(\tilde X^j)\ar[d]^{\pi_j}\\
F_{j+1}\ar[r]_\rho&Z_j(X^j). }
\]
The vertical map $\tilde F_{j+1}\to F_{j+1}$ actually factors through $ S\otimes_{\tilde S}\tilde F_{j+1}$, which is a free $S$-module and has a splitting cokernel. The existence of the $S$-module homomorphism $ S\otimes_{\tilde S}\tilde F_{j+1}\to F_{j+1}$ follows from the fact that $ \pi_j(\tilde \rho (\tilde F_{j+1}))\subseteq \rho(F_{j+1})$ since $\pi_j(\tilde C_j)\subseteq C_j\cup \{0\}$. 
 
Then we use the universal mapping property to get a differentially graded algebra homomorphism   $ \on{Sym}_{\tilde S}(\tilde F_{j+1}[j+1]) \to \on{Sym}_{ S}( F_{j+1}[j+1])$, which sends basis elements to basis elements or zero with a fixed choice of monomial basis (by choosing total orderings in $\tilde C_j$ and $C_j$ so that $\pi_j$ preserves the total ordering.  We now obtain the differential graded algebra homomorphism $ \tilde X^{j+1}=\tilde X^j\otimes \on{Sym}_{\tilde S}(\tilde F_{j+1}[j+1]) \stackrel{\pi_{j+1}}{\to} X^{j}\otimes \on{Sym}_{ S}( F_{j+1}[j+1])=X^{j+1}$, together with the choices of the basis $\cal B(\tilde X^{j+1})$ and $\cal B(X^{j+1})$. The rest follows the construction in the theorem. 
\end{proof}

\begin{corollary} \label{cor:ext-hom}
If $ \pi:\tilde S\to S$ is  a surjective homomorphism of commutative ring with $ \pi(\tilde I)\subseteq I$ and such that $\tilde S/ \tilde I \cong S/I$, then there is a homomorphism of algebras
\[\pi^\#: \on{Ext}_S^*(S/I, S/I)\to \on{Ext}_{\tilde S}^*(\tilde S/\tilde I, \tilde S/\tilde I).\]
\end{corollary}

We leave the proof to the readers by simply mentioning the following natural homomorphisms of abelian groups
\[ \on{Hom}_S(X_n, S/I)\to \on{Hom}_S(S\otimes_{\tilde S}\tilde X_n, S/I)\to \on{Hom}_{\tilde S}(\tilde X_n, S/I)\to \on{Hom}_{\tilde S}(\tilde X_n, \tilde S/\tilde I).
\]
The Yoneda composition defined in \eqref{eq:Yoneda-composition} is clear using the fact that the restriction functor $ \pi^*: S\Mod\to \tilde S\Mod$ is exact. 

\subsection{Complete intersections}\label{sec:4.4}
A free resolution can be obtained explicitly in a particular case. A sequence of elements $(c_1, c_2, \dots, c_r)$ in $S$ is said to be a \textbf{regular sequence} if $c_1$ is not a zero divisor in $S$, and if, for each $i$, $1 \leq i < r,$ the residue class of $c_{i+1}$ is not a zero divisor in the residue class ring $S/ \langle c_1, \dots , c_i \rangle $. The quotient $S/ \langle c_1, c_2, \dots, c_r \rangle$ is then called a \textbf{complete intersection}. The next result is also due to J. Tate (\cite[Theorem 4]{Tate}).

\begin{theorem}\label{TateTheorem}
Let $t_1, \dots, t_n$ and $c_1, \dots, c_r$ be regular sequences of elements of $S$ such that the ideal $I= \langle c_1, \dots, c_r \rangle$ generated by the $c_i$'s is contained in the ideal $\frak m= \langle t_1, \dots, t_n \rangle$ generated by the $t_i$'s. Write $c_j =\sum_{i=1}^n c_{j, i}t_i$, with $c_{j, i} \in S$. Set $\bar{S}=S/I$ and $\overline{\frak m}=\frak m/I$, and let $\overline{c_{j, i}}$ and $\overline{t_i}$ denote images  of $c_{j, i}$ and $t_i$ in $\bar S$. Then the algebra
\begin{align*}
X=\bar{S} \langle T_1, \dots, T_n, S_1, \dots, S_r \rangle =\bar{S}[S_1, \dots, S_r]\otimes_{\bar{S}}\wedge_{\bar{S}}[T_1, \dots, T_n]
\end{align*}
with the elements in $ \bar S$ in degree $0$, the $T_i$'s of degree $1$, the $S_j$'s of degree 2, and the differential $ \partial $ defined by 
\begin{align}\label{eq:ci-differential}
\partial_1 T_i=\overline{t_i}, \quad \partial_2 S_j =\sum_{i=1}^n \overline{c_{j, i}} \, T_i,
\end{align}
is acyclic, and therefore yields a free resolution of the $\bar{S}$-module $\bar{S}/\overline{\frak m}$.
\end{theorem} 

If we follow the construction of the Tate resolution to the map $ \bar S\to \bar S/\overline{\frak m}$, taking $ X^1=\wedge_{\bar S} [ T_1, \dots, T_n ]$ corresponding to the generators $ \{\bar t_1,\dots, \bar t_n\}$ of $\overline{\frak m}$, then the kernel of $\partial: X^1_1\to \bar S$ is generated as $ \bar S$-module by $\sum_{i=1}^n\overline{c_{j, i}} \, T_i$ ($1\leq j\leq r$) and $ \partial (T_iT_{i'})$.  Let the $S_i$'s be the basis of $ F_2$ and we form $X^2=X^1\otimes_{\bar S}\on{Sym}_{\bar S}(F_2[2])$. 
The above theorem says that $X=X^2$ for the ring $\bar S$ which is a complete intersection.

\subsection{Computing the extension groups $\on{Ext}^{*}_{S/I}(S/\frak m, S/\frak m)$}\label{sec:Yoneda_prod}
Consider the cochain complex $ \on{Hom}_{\bar{S}}(X, S/\frak m)$ with $ \on{Hom}^i_{\bar{S}}(X, S/\frak m)=\on{Hom}_{\bar{S}}(X_i, S/\frak m)$ consisting of all homogeneous $S$-module homomorphisms of degree $-i$ with $ S/\frak m$ regarded as a graded $ \bar{S}$-module concentrated in degree $0$. The differential $d^i: \on{Hom}^i_{\bar{S}}(X, S/\frak m)\to \on{Hom}^{i+1}_{\bar{S}}(X, S/\frak m)$ is defined by $ d^i(f)=f\circ \partial_{i+1}$. 
  We have  
\[\on{Ext}^{*}_{S/I}(S/\frak m, S/\frak m)=H^*(\on{Hom}_{\bar{S}}(X, S/\frak m)).\] 
Note that $\on{Ext}^{*}_{S/I}(S/\frak m, S/\frak m)$ are actually $\bar S/\overline{\frak m}=S/\frak m$-modules.   Suppose $ \partial_i(X_i)\subseteq \overline{\frak m} X_{i-1}$ for all $i$. Then the differential $d^{i-1}: \on{Hom}_{\bar{S}}(X_{i-1}, S/\frak m)\to \on{Hom}_{\bar{S}}(X_{i}, S/\frak m) $ is $0$. Thus we have $\on{Ext}^{i}_{\bar S}(S/\frak m, S/\frak m)=\on{Hom}_{\bar{S}}(X_i, S/\frak m)$ as $ S/\frak m$-module.  Note that
\[\on{Hom}_{\bar{S}}(X_i, S/\frak m)\cong \on{Hom}_{S/\frak m}(S/\frak m\otimes_{\bar S}X_i, S/\frak m).
\]
Since $ S/\frak m\otimes_{\bar S}X=S/\frak m[S_1,\dots, S_r]\otimes_{S/\frak m}\wedge_{S/\frak m}[ T_1, \dots, T_n] $ as differential graded $ S/\frak m$-algebras with zero differentials, then we have
\begin{align}\label{eq:ci-extension}
\on{Ext}^{*}_{S/I}(S/\frak m, S/\frak m)=\bigoplus_{k=0}^\infty (\on{Sym}^k_{ S/\frak m}( S_1,\dots, S_r))^*\otimes_{S/\frak m}\bigoplus _{l=0}^{n}(\wedge^l_{S/\frak m} [ T_1, \dots, T_n])^*
\end{align}
as graded $ S/\frak m$-modules, and each homogeneous component is free of finite rank. Here $\{T^*_1, \dots, T^*_n\}$ is the dual basis of $\{T_1, \dots, T_n\}$ of the free $S/\frak m$-module $\frak m/\frak m^2$ and $ \{S_1^*, \dots, S_r^*\}$ is the dual basis of $\{S_1, \dots, S_r\}$ of the $ S/I$-module $ I/I^2$.

 If  $c_j \in \frak m^2$, then $c_{j, i}\in \frak m$ for all  $i$ (and thus $\overline{c_{j,i}}=0$ in $S/\frak m$). Then applying $\partial (T_i)$ and $\partial (S_j)$ defined in the theorem, we have $\partial (S_1^{n_1}\cdots S_r^{n_r}T_1^{e_1}\cdots T_n^{e_n})\in \frak m X$. Here $ (n_1, \dots, n_r)\in \bb N^r$ and $ (e_1, \dots, e_n)\in \{0, 1\}^n$.  Thus under the assumption $ I\subseteq \frak m^2$, we see that \eqref{eq:ci-extension} holds.

\subsection{The Yoneda product}
We now describe the Yoneda product. For the chain complex $X_*$, we write $X_{\geq p}=\oplus_{i\geq p}X_i$. Then $X_{\geq p}$ is an ideal of $X$ and also a chain complex. However, $X_{\geq p}$ not a subcomplex of $X_i$ and thus not a differential graded ideal. It is rather a quotient complex of $X$ with projection $X \to X_{\geq p}$. We make $ X_{\geq p}$ a $X$-$X$-bimodule via left and right multiplications which satisfy  
\[ xv=(-l)^{\deg(v)\deg(x)}vx
\] for homogeneous $ v\in X_{\geq p}$ and $x \in X$. 

We also define the chain complex $X[p]$ with $(X[p])_n = X_{n-p}$ and $(\partial[p])_n = (-1)^{p}\partial_{n-p}$. Any $ \xi: X_p\to S/\frak m$ can be extended to a chain map $ \hat \xi :X_{\geq p}\to X[p]$ with $ \hat \xi_p: X_p \to X_0=S$ lifting $\xi$. Thus we can view $ \hat \xi$ as chain map $X \to X[p]$ via the quotient $X \to X_{ \geq p}$.  Hence $ \hat \xi_n\circ \partial_{n+1}=(-1)^p\partial _{n+1-p} \circ \hat \xi_{n+1}$ with $ \hat\xi_n: X_{n}\to X_{n-p}$.  If $ \zeta: X_q\to S/\frak m$ is another $\bar S$-module homomorphism with a lifting $\hat \zeta: X \to X_{\geq q}\to X[q]$ then we can define the composition of  chain maps
\[X \stackrel{\hat\xi}{\to}X[p]\stackrel{\hat\zeta[p]}{\to} X[p+q].\] 
Thus the Yoneda product $\zeta \xi: X_{p+q}\to S/\frak m$ is the composition of $ (\hat \zeta[p] \hat\xi)_{p+q}=\hat \zeta_q\hat \xi _{p+q}:X_{p+q}\to S$ and  $S \to S/ \frak m$, which is exactly $ \zeta\circ \hat \xi_{p+q}$. We remark that the product $\zeta\xi$ differs from the Yoneda composition $\zeta\circ \xi$ defined in \eqref{eq:Yoneda-composition}. They are related by $\zeta\circ \xi=(-1)^{p q}\zeta\xi$.  This construction is more convenient in the derived category sense.

The trouble in computing the Yoneda product is that the are many different liftings $ \hat \xi$ for each $ \xi$. However, the algebra structure on $X$ can help to find $\hat\xi_m$ so that it satisfies the derivation property
 \begin{align}\label{eq:derivation}
\hat\xi _{|u|+|v|}(u v)=\hat\xi_{|u|}(u)v+(-1)^{|u|p} u \hat\xi_{|v|}(v)
\end{align}
for all $ u, v \in X_{\geq p}$ for which $\hat\xi_{|u|}(u)$ and $ \hat\xi_{|v|}(v)$ have been defined. It can be directly verified that $ (-1)^{p}\partial \hat\xi(u)= \hat\xi (\partial (u))$ and $ (-1)^{p}\partial \hat\xi(v)= \hat\xi (\partial (v))$ implies $ (-1)^{p}\partial \hat\xi(uv)= \hat\xi (\partial (uv))$. 

We then use the fact that 
$X_m=\sum_{i=1}^nT_iX_{m-1} +\sum_{j=1}^r S_j X_{m-2}$  for any $m\geq 2$ to reduce the computation to $ \hat\xi(T_ix)$ and $ \hat\xi(S_ix)$.  For lower degrees, we will need to apply the following 
\begin{align}\label{eq:multi-differential}
(-1)^{p}\partial \hat\xi(uv)=\hat\xi (\partial(uv))=\hat\xi (\partial(u)v)+(-1)^{|u|}\hat\xi(u\partial(v))
\end{align}
for any $u, v\in X$ homogeneous to compute the $\hat\xi_m$, together with $\partial T_i=t_i$ and $\partial S_j=\sum_{i=1}^n \overline{c_{j, i}} \, T_i$. 

Recall that $T_i^*: X_1\to X_0$ is defined by $T_i^*(T_{i'})=\delta_{i, i'}$   and $ S_j^*: X_2\to X_0$ is defined by $S_j^*(S_{j'})=\delta_{j, j'}$ and $S_j^*(T_iT_{i'})=0$. 

Let  $\alpha_i=T_i^*\in \on{Ext}^{1}_{S/I}(S/\frak m, S/\frak m) $ and $ \beta_j=S_i^*\in \on{Ext}^{2}_{S/I}(S/\frak m, S/\frak m)$.  To calculate the multiplication relations among them, we first compute $  \hat \alpha_i$ and $ \hat\beta_j$. 
 We write $(\hat \alpha_i)_2(S_j)=-\sum_{l=1}^n \bar s_{j}^{i, l}T_l\in X_1$ for some $ \bar s_{j}^{i, l}\in \bar S$.  The equation $ \hat\alpha_i(\partial_2(S_j))=-\partial_1(\hat\alpha_i(S_j))$ is equivalent to the equation 
 \begin{align}\label{eq:c-s-equation}\bar c_{j,i}=\sum_{l=1}^n \bar s_j^{i, l} \bar t_l
 \end{align}
 having solutions $ \bar s_{j}^{i, l}$ in $\bar S$ for all $1 \leq j \leq r$ and $1 \leq i, l\leq n$. 
 Assume $ c_j\in \frak m^2$ and we choose $ s_{j}^{i, l}\in S$ such that write $ c_j=\sum_{i, l=1}^{n}s_{j}^{i, l}t_it_l$ in $ S$. Then $ \hat\alpha_2(S_j)$ can be defined. 
  Using \eqref{eq:derivation}, $ (\hat\alpha_i)_m$ can be defined for all $m$. Here we can compute $ \xi\alpha_i$ for all $ \xi\in \on{Ext}_{S/I}^q(S/\frak m, S/\frak m)$. Let us list $(\hat\alpha_i)_2$ and $ (\hat\alpha_i)_3$ for the convenience of later computations. 
 \[ (\hat\alpha_i)_2 (T_pT_q)=\delta_{i, p}T_q-\delta_{i, q}T_p\]
 \begin{align*}
 (\hat\alpha_i)_3 (T_lT_pT_q)&=\delta_{i, l}T_pT_q-\delta_{i, p}T_lT_q+\delta_{i, q}T_lT_p,\\
 (\hat\alpha_i)_3 (T_lS_j)&=\delta_{i, l}S_j+\sum_{k=1}^n \bar s_{j}^{i, k}T_lT_k.
 \end{align*}
 Thus, by defining
\begin{align*}
\begin{array}{rccl}
 (T_kT_i)^* : & X_2 & \longrightarrow& S/\frak m \\
  & T_pT_q  & \mapsto & \begin{cases} 1 \text { if } p=k, q=i ,\\
-1 \text{ if } p=i, q=k,\\
0 \text { otherwise, }
 \end{cases} \\
 & S_j & \mapsto & 0,
\end{array}
\end{align*}
we have
 \begin{align}\label{eq:alpha_product}\alpha_k\alpha_i=-\sum_{j=1}^r \bar s_j^{i, k}\beta_j-(T_kT_i)^*,
 \end{align} 
 and 
 \[ \alpha_k\alpha_i+\alpha_i\alpha_k=-\sum_{j=1}^{r}(\bar s_{j}^{i, k}+ \bar s_{j}^{k,i})\beta_j.\]
 We also have $\beta_j\hat\alpha_i(T_lT_pT_q)=0$ for all $ p, q, l$ distinct and $\beta_j\hat\alpha_i (T_pS_k)=\delta_{i, p}\delta_{j, k}$ so
 \[\beta_j\alpha_i=(T_iS_j)^*.\]
  
 To compute $ \beta_{j}\beta_{j'}$ and $\alpha_i\beta_j$, we need to compute $(\hat\beta_j)_3$ and $(\hat\beta_j)_4$ first. Using \eqref{eq:derivation}, we see that 
  \begin{align*}
  \hat\beta_j(xT_lT_k)&= \hat\beta_{j}(x)T_lT_k,\\
  \hat\beta_j(xS_p)&=\hat\beta_{j}(x)S_p+\delta_{j, p} y.
  \end{align*}
Then $(\hat\beta_j)_m$ can be computed. In particular, we can verify that $(\hat\beta_j)_m $ on the commutative algebra $ (S/I)[S_1, \dots, S_r]$ is given by the partial derivative $ \frac{\partial}{\partial S_j}$ and $ 0$ on $XT_kT_l$.  By writing $ (\hat\beta_j)_3(S_pT_i)=\sum_{l=1}^na_{j, p}^{i,l}T_l\in X_1$ and using \eqref{eq:multi-differential}, we see that 
\[ (\hat\beta_j)_3(S_pT_i)=\delta_{j, p}T_i=\hat\beta_j(S_p)T_i.
\] 
Thus $\hat\beta_i: X\to X[2]$ is a right $\wedge_{S/I} [T_1, \dots, T_n] $-module homomorphism.   Thus $\alpha_i\beta_j=(T_iS_j)^*=\beta_j\alpha_i$. By direct computation we show that $\beta_{j'}\beta_j=\beta_j\beta_{j'}$.  It follows from \eqref{eq:ci-extension} and the fact
 \[(S_1^{a_1}\cdots S_r^{a_r})^* =\frac{1}{(a_1!) \cdots (a_r!)} (S_1^*)^{a_1}\cdots (S_r^*)^{a_r}\]
 that  the subalgebra $S/\frak m \langle \beta_1, \dots, \beta_r\rangle $ is isomorphic to $S/\frak m[S_1^*, \dots, S_r^*]$.
  
 Let $ A=  S/\frak m[S_1^*, \dots, S_r^*]$ be the graded commutative algebra over $ S/\frak m$ with $ S_i^*$ in degree $2$. We consider the subalgebra  $A\langle \alpha_1, \dots, \alpha_n\rangle \subseteq  \on{Ext}^{*}_{S/I}(S/\frak m, S/\frak m)$.  By computing higher order products among $ \alpha_1, \dots, \alpha_n$ similar to computation of \eqref{eq:alpha_product}, we can see that $ (T_{1}^{e_1}\cdots T_{n}^{e_n})^*\in A\langle \alpha_1, \dots, \alpha_n\rangle $. By \eqref{eq:ci-extension} we get 
\begin{align}A\langle \alpha_1, \dots, \alpha_n\rangle =  \on{Ext}^{*}_{S/I}(S/\frak m, S/\frak m).
\end{align} 
  
  Let $\on{Cl}_A\langle x_1, \dots, x_n\rangle $ be the graded Clifford algebra (with $ x_i$ in degree 1) over $A$ with respect to the symmetric $A$-bilinear form 
  \begin{align}
   ( x_i, x_{i'})=-\sum_{j=1}^{r}(\bar s_{j}^{i,i'}+ \bar s_{j}^{i',i})\beta_j.
  \end{align}
  Then there is a natural surjective homomorphism of graded algebras $ \on{Cl}_A\langle x_1, \dots, x_n\rangle\to \on{Ext}^{*}_{S/I}(S/\frak m, S/\frak m)$. By using \eqref{eq:ci-extension} and counting the ranks of the free $A$-modules we get the following presentation of the algebra $\on{Ext}^{*}_{S/I}(S/\frak m, S/\frak m)$.

 Let us summarize the above discussion as follows. 

\begin{theorem} \label{thm:ci-ext-alg}
Let $t_1, \dots, t_n$ and $c_1, \dots, c_r$ be regular sequences of elements of a commutative ring $S$ containing $\bb Q $ and $I=\langle c_1, \dots, c_r\rangle$ and $\frak m=\langle t_1, \dots, t_n\rangle$ ideals of $S$. 
Assume  that $ I\subseteq \frak m^2$ and write $c_{j}=\sum_{i=1}^{n}\sum_{l=1}^ns_{j}^{i, l}t_it_l$ and $ \bar S=S/I$. Then $\on{Ext}^{*}_{ \bar S}(S/\frak m, S/\frak m)$ is isomorphic to the graded $ S/\frak m$-algebra
\[ \on{Ext}^{*}_{  S/I}(S/\frak m, S/\frak m)=(S/\frak m)[\beta_1,\dots, \beta_r]\langle \alpha_1, \dots, \alpha_n\rangle/\cal I.\]
Here $(S/\frak m)[\beta_1,\dots, \beta_r]$ is the polynomial with variables $ \beta_1, \dots, \beta_r$ with $\deg(\beta_j)=2$ and $\deg(\alpha_i)=1$. The ideal $\cal I$ is generated by the relations $ \alpha_i\alpha_k+\alpha_k\alpha_i +\sum_{j=1}^r(\bar s_{j}^{i,k}+\bar s_{j}^{k,i})\beta_j$ over the polynomial ring $(S/\frak m)[\beta_1,\dots, \beta_r]$.  \end{theorem}

\begin{corollary} \label{cor:deg3-ci-Ext algebra}
Under the assumption of Theorem~\ref{thm:ci-ext-alg}, if we further assume that $ I\subseteq \frak m^3$, then $ \overline{c_{j, i}}=0$ in $ S/\frak m$ and 
\[\on{Ext}^{*}_{  S/I}(S/\frak m, S/\frak m)=(S/\frak m)[\beta_1,\dots, \beta_r]\otimes_{S/\frak m}\wedge_{S/\frak m}[\alpha_1, \dots, \alpha_n]\]
is a commutative differential graded $ S/\frak m$-algebra with  zero differential. 
\end{corollary}

\begin{proof}
The commutativity is equivalent to $s_{j}^{i,k}+s_{j}^{k,i}\in\frak m$. The assumption $I\subseteq \frak m^3 $ implies that $ s_{j}^{i,k}\in \frak m$. 
\end{proof}

\subsection{The cohomology of the modified complete intersection ring}\label{ComputeH*(Rtilde)}

Let us set $S=\bfk[t_1,\dots, t_n]$ the polynomial algebra with standard grading and  $R=S/ \langle c_1,\dots, c_r \rangle $ where $c_1, \dots, c_r$ are homogeneous polynomials of degree at least $2$  such that they form a minimal set of generators of the ideal $ \langle c_1,\dots, c_r \rangle$. By \textbf{minimal}, we mean that for any $ 1\leq i\leq r$, $ c_i\not\in\sum_{j\neq i} S c_j$.   
In particular, if all $ c_1, \dots, c_r$ are homogeneous of the same degree and are $\bfk$-linearly independent, then they form a minimal set of generators of the ideal they generate.

Since the $c_i$'s have degree at least $2$, in what follows, for all $1 \leq j \leq r$, we will write $c_j=\sum_{i=1}^n c_{j, i}t_i=\sum_{i, l=1}^n s_j^{i,l}t_{i}t_l$, with $s_{j}^{i, l} \in S=\bfk[t_1,\dots, t_n]$. The matrix $(s_{j}^{i,l})$ for each $j$ is  not unique and we will fix them for the convenience of computations in what follows. Since the sequence $\{c_1, \dots, c_r\}$ is not necessarily a regular sequence, we consider the polynomial algebra $ \tilde S=\bfk[\tilde t_1, \dots \tilde t_n, \tilde t_{n+1},\dots, \tilde t_{n+r}] $ and define the quotient map $\pi: \tilde S\to S$ by setting $ \tilde t_i \mapsto t_i$ for $i\leq n$ and $\tilde t_i \mapsto 0$ for $i> n$.  We then set $ \tilde c_j=c_j -\tilde t_{n+j}^3$ for all $1 \leq  j \leq r$. Here we have simply identified $S$ with the subalgebra $\bfk[\tilde t_1, \dots, \tilde t_n]$ of $ \tilde S$ under the map $ t_i\to \tilde t_i$, and will use the same notation for the elements of $S$ as elements of $ \tilde S$. When this occurs, it will be clear from context. Thus $c_j$ is regarded as an element of $ \tilde S$ in the expression of $ \tilde c_j$.  Then $ \{\tilde c_1, \dots, \tilde c_r\}$ is a regular sequence in $ \tilde S$. We see that the ring $ \tilde R=\tilde S/\langle \tilde c_1, \dots, \tilde c_r\rangle $ is a complete intersection, and we will use $ \pi: \tilde R\to R$ to denote the induced quotient map.  Let $ \frak m\subseteq S$ and $ \tilde{\frak m} \subseteq \tilde S$  be the ideals generated by $ \{t_1, \dots, t_n\}$ and $ \{\tilde t_1, \dots, \tilde t_{n+r}\} $ respectively. Then $ \pi(\tilde{\frak m})=\frak m$. Note that $ \bfk=S/\frak m=\tilde S/\tilde{\frak m}$ and this induces augmentations $ R\to \bfk$ and $\tilde R\to \bfk$. We will use $ \overline{\frak m}\subseteq R$ and $ \overline{\tilde{\frak m}}\subseteq \tilde R $ to denote the corresponding augmentation ideals.  

Similarly, we write $\tilde c_j=\sum_{i, l=1}^{n+r} \tilde s_{j}^{i,l}\tilde t_i\tilde t_l\in \tilde S$.  Then, for each $j$, we have 
\begin{align}
 \tilde s_j^{i,l} =\begin{cases} s_{j}^{i, l} & \; \text { if } i, l\leq n ,
\\
-\tilde t_i & \; \text{ if } i=l >n,\\
0& \; \text { otherwise. }
 \end{cases}
\end{align}

\delete{Geometrically, the presentation $ \bfk[t_1, \dots, t_n]/\langle c_1, \dots, c_r\rangle $ defines a closed imbedding of $ \on{Spec}(R)\subseteq \bb A^n$. While $ \on{Spec}(\widetilde R)\subseteq \bb A^{n+r}$ is a closed and $ \bb A^n\subseteq \bb A^{n+r}$ are closed subschemes, then $ \on{Spec}(R)= \on{Spec}(\widetilde R)\cap \bb A^n$ in $ \bb A^{n+r}$. The ideals $ \frak m$ correspond to the point $0$ in each of the schemes. 
}

Then $\pi (\overline{\tilde{\frak m}})
=\overline{\frak m}$.
The $\bfk$-algebra morphism $\pi: \widetilde R\to R$  induces a graded algebra homomorphism  $\pi^{\#} : \on{Ext}^*_R(\bfk, \bfk) \to  \on{Ext}^*_{\tilde R}(\bfk, \bfk) $ by Corollary~\ref{cor:ext-hom}. In the following, we describe the image of $ \pi^{\#}$.  

By following the construction (or the proof) of Theorem~\ref{thm:compatible_tate_resolution}, we construct
\begin{align}
\tilde X^2&=\tilde R[\tilde S_1, \dots, \tilde S_r]\otimes_{\tilde R} \wedge_{\tilde R} [ \tilde T_1, \dots, \tilde T_{n+r}],\label{eq:tilde-X2}\\
X^2&= R[ S_1, \dots,  S_r]\otimes_R \wedge_{ R} [  T_1, \dots,  T_{n} ]. \label{eq:X2}
\end{align}
The map $ \pi: \tilde X^2\to X^2$ is defined by applying $R\otimes_{\tilde R}-$ and sending $ \tilde T_i$ to $0$ for all $i>n$. 
Since $ \tilde R$ is a complete intersection, then $ \tilde X=\tilde X^2$ by Theorem~\ref{TateTheorem}.
Note that in the inductive step of constructing $ \tilde X^{j+1}$  in the proof of Theorem~\ref{thm:compatible_tate_resolution} for $ j\geq 2$, the choice of $ \tilde C_j$ can be empty since $ B_j(
X^j)=Z_j(X^j)$. Thus the differential algebra homomorphism $ \tilde X\to X$ has its image onto $X^2 \subseteq X$. This induces a surjective graded $\bfk$-module homomorphism 
\begin{align}
\on{Hom}_{R}(X, \bfk)\stackrel{\iota^*}{\twoheadrightarrow } \on{Hom}_{R}(X^2, \bfk)\hookrightarrow \on{Hom}_{\tilde R}(\tilde X^2, \bfk)
\end{align}
since  homogeneous components of $X$ and $ X^2$ are free $R$-modules of finite rank and $ X^2_m$ is  direct summand of $X_m$.  

Since $ c_1, \dots, c_r\in \mathfrak m^2$ (resp. $ \tilde c_1, \dots, \tilde c_r\in \tilde{\mathfrak m}^2$), then $ \partial X^2_m\subseteq \overline{\mathfrak m} X^2_{m-1}$ (resp. $ \partial \tilde X^2_m\subseteq \overline{\tilde{\frak m}} \tilde X^2_{m-1}$). Thus we have 
\begin{align}\on{Ext}^{*}_{R}(\bfk, \bfk) \stackrel{\iota^\#}{\to} H^*(\on{Hom}_{R}(X^2, \bfk))=\on{Hom}_{R}(X^2, \bfk)\hookrightarrow \on{Ext}^{*}_{\tilde R}(\bfk, \bfk)
\end{align}
as graded $ \bfk$-modules. The composition gives a graded $\bfk$-algebra homomorphism 
\begin{align}
\on{Ext}^{*}_{R}(\bfk, \bfk) \stackrel{\pi^{\#}}{\to} \on{Ext}^{*}_{\tilde R}(\bfk, \bfk)
=\bfk[\tilde\beta_1,\dots, \tilde\beta_r]\langle \tilde \alpha_1, \dots, \tilde\alpha_{n+r}\rangle/\cal I
\end{align}
determined by  Corollary~\ref{cor:ext-hom} and Theorem~\ref{thm:ci-ext-alg}.

The $\bfk$-subspace $ \on{Hom}_{R}(X^2, \bfk)$ in $\on{Ext}^{*}_{\tilde R}(\bfk, \bfk)$ is exactly the subalgebra whose generators are $ \{\tilde \beta_{1}, \dots, \tilde\beta_r\}$  and $ \{ \tilde \alpha_1,\dots, \tilde \alpha_n\}$ by \eqref{eq:tilde-X2} and \eqref{eq:X2}. This $\bfk$-subalgebra is presented by
\begin{align} \bfk[\tilde\beta_1,\dots, \tilde\beta_r]\langle \tilde \alpha_1, \dots, \tilde\alpha_{n}\rangle/\cal J
\end{align}
with 
\[\cal J=\langle 
 \alpha_i\alpha_l+\alpha_l\alpha_i+\sum_{j=1}^r (\bar s_j^{i, l}+\bar s_j^{l, i})\beta_j\rangle
 \]
 with $ 1\leq i, l\leq n$.

Our goal is to prove that $\iota^\#:\on{Ext}^{*}_{R}(\bfk, \bfk) \to \on{Hom}_{R}(X^2, \bfk)$ is surjective by constructing cocycles $\alpha_i$ and $\beta_j$ in $\on{Hom}_{R}(X, \bfk)$ such that $ \pi^\#(\alpha_i)=\tilde \alpha_i $ and $ \pi^\#(\beta_i)=\tilde \beta_j$ (for $ i=1, \dots, n$ and $ j=1, \dots, r$). The surjectivity of $\iota^\#$ will follow since $ \pi^\#$ is an algebra homomorphism.

It follows from the construction of Theorem~\ref{thm:tate-resolution-ideals} that we have $ X_i=X^2_i$  and 
\[\on{Hom}_{R}(X_i, \bfk)=\on{Hom}_{R}(X^2_i, \bfk)
\]
for  $i\leq 2$.  
We note that $Z_i(X)=Z_i(X^2)$ for $ i\leq 2$ and $B_i(X)=B_i(X^2)$ for $ i< 2$. Therefore $\on{Ext}_R^1(\bfk,\bfk)=\on{Hom}_R(X^2_1, \bfk)$ has a $\bfk$-basis given by the $\alpha_i=T_i^*$ for $ i=1, \dots, n$ and $ \pi^\#(\alpha_i)=\tilde\alpha_i$.  However,
we only have $  Z_2(X)=B_2(X)\supseteq B_2(X^2)$. We need the following lemma. 

\begin{lemma}
If $\sum_{j =1}^{r}\overline{s_j} S_j+\sum_{1\leq i <m\leq n}\overline{a_{i, m}} \, T_iT_m \in Z_2(X^2)$ (with $ s_j, a_{i,m}\in S$), then all $ s_1, \dots,  s_r \in \frak m$. Consequently  $S_j^*\in \on{Hom}_{R}(X_i, \bfk)$ is a cocyle. 
\end{lemma}

\begin{proof}
We apply $ \partial_2$ to the element in the statement and we get 
\begin{align*}
0 &\displaystyle  =\sum_{j =1}^r\overline{s_j} (\sum_{i=1}^n\overline{c_{j, i}} \, T_i)+\sum_{1\leq i <m\leq n} \overline{a_{i, m}} (\bar t_iT_m-\bar t_mT_i), \\
 & \displaystyle = 
 \sum_{i=1}^n(\sum_{j =1}^r\overline{s_j} \, \overline{c_{j, i}}+\sum_{m<i}\bar t_m \overline{a_{m, i}}-\sum_{i<m}\bar t_m \overline{a_{i, m}})T_i.
\end{align*}
Since $T_i$ are basis elements of a free module, we get that for all $1 \leq i \leq n$:
\begin{align*}
\sum_{j=1}^{r}\overline{s_j} \, \overline{c_{j, i}}+\sum_{m<i}\bar t_m \overline{a_{m, i}}-\sum_{i<m}\bar t_m \overline{a_{i, m}}=0 \text{ in } R.
\end{align*}
Writing in $S$ we have, for $1\leq i \leq n$,
\begin{align*}
\sum_{j=1}^{r}s_j  c_{j, i}+\sum_{m<i} t_m a_{m, i}-\sum_{i<m} t_m a_{i, m}=\sum_{l=1}^r s_{i,l} c_l.
\end{align*}
for $s_{i, l}\in S$. By multiplying the expression corresponding to $i$  by $t_i$ and then summing over $i$, the last two terms cancel each other and we get 
\begin{align*}
\sum_{j =1}^r s_j c_j=\sum_{l=1}^r\sum_{i=1}^n t_i s_{i, l}c_l
\end{align*}
in $\bfk[t_1, \dots, t_n]$. If one of the $s_j\not \in \frak m$, say $j=1$, we can assume the constant term of $s_1$ is $1$ and write $ s_1=1+s'$ with  $ s'\in \frak m$. Therefore 
\[ (1+s'-\sum_{i=1}^n t_is_{i, 1})c_1 \in \sum_{j=2}^rS c_j.\]
Since $c_1, \dots, c_r$ are homogeneous, by applying the projection map $S\to S_d$ to the homogeneous component of degree $ d=\deg (c_1)$, we get $ c_1=\sum_{j=2}^r b_j c_j$ with $\deg (b_j)+\deg (c_j)=d$ if $ b_j c_j\neq 0$ for homogeneous $b_j \in S$. This contradicts  the minimality condition of $\{c_1, \dots, c_r\}$. 

The second part of the statement follows from the fact that $d^2(S_j^*)=S_j^*\circ \partial_3$ and $ \partial_3 X_3=B_2(X)=Z_2(X^2)$. 
\end{proof}

We now summarize the result as follows.
 
\begin{theorem}\label{thm:quotientofExt} 
Let  $ R=\bfk[t_1,\dots, t_n]/I$. Suppose that  $ \{ c_1, \dots, c_r\}$ is  a minimal set of generators of homogeneous elements of degree at least 2 of the ideal $ I$. Write $ c_{j}=\sum_{i, i'=1}^n s_{j}^{i, i'}t_it_{i'}$. Then there is a surjective homomorphism of graded algebras
\begin{align}
\on{Ext}^*_{R}(\bfk, \bfk)\twoheadrightarrow \bfk[\beta_1, \dots, \beta_r]\langle \alpha_1, \dots, \alpha_n\rangle/\cal J.
\end{align}
Here $ \cal J=\langle \alpha_i\alpha_{i'}+\alpha_{i'}\alpha_i+\sum_{j=1}^{r}(\bar s_{j}^{i,i'}+ \bar s_{j}^{i',i})\beta_j \ | \ 1 \leq i, i' \leq n \rangle$. 
\end{theorem}
 
\begin{corollary}
If $c_j \in \frak m^3$ ($1 \leq j \leq r$) in Theorem~\ref{thm:quotientofExt}, then there is a surjective homomorphism of graded algebras 
\begin{align}
\on{Ext}^*_{R}(\bfk, \bfk)\twoheadrightarrow \bfk[\beta_1, \dots, \beta_r]\otimes_{\bfk}\wedge_{\bfk}[ \alpha_1, \dots, \alpha_n].
\end{align}
In particular, there is also a surjective homomorphism of graded algebras
\[\on{Ext}^*_{R}(\bfk, \bfk) \twoheadrightarrow \bfk[\beta_1, \dots, \beta_r].\]
\end{corollary}

A direct consequence of the results above is the following theorem, which is the one we mentioned at the beginning of this section.

\begin{theorem}\label{SliceofMaxH(R)ab}
Set $R=\bfk[t_1,\dots, t_n]/ I$ and $\{c_1, \dots, c_r\}\subseteq I$ a minimal set of homogeneous generators of degree at least $3$.  We consider the $R$-module $\bfk=R/ \langle \bar t_1, \dots,\bar t_n \rangle$.  Then $ \on{Spec}(\on{Ext}^*_{R}(\bfk, \bfk)^{ab})$ contains a closed subscheme $\bb A^r$.
\end{theorem}

\section{The cohomological variety of the  Virasoro vertex operator algebra}\label{sec:Vir}
The Virasoro Lie algebra $\mathcal{L}$ is spanned by the set $\{L(n)\}_{n \in \mathbb{Z}} \cup \{\mathbf{c} \}$ such that $\mathbf{c} $ is central, and with the relations:
\[
[L(m), L(n)]=(m-n)L(m+n)+\frac{1}{12}(m^3-m)\delta_{m+n,0}\mathbf{c}.
\]
We write $\mathcal{L}_{\geq -1}=\bigoplus_{n \geq -1} \cc L(n) \oplus \cc \mathbf{c}$. Let $c$ be a complex number and let $\cc_c$ denote the one dimensional $\mathcal{L}_{\geq -1}$-module $\cc$ such that $\bigoplus_{n \geq -1} \cc L(n)$ acts as zero and $\bold{c}$ acts as $c$. The universal Virasoro vertex operator algebra of central charge $c$ is defined as the induced module
\begin{align*}
V_{Vir}(c,0)= U(\mathcal{L}) \otimes_{U(\mathcal{L}_{\geq -1})} \cc_c.
\end{align*}
The vacuum vector is $\mathbf{1}= 1 \otimes 1$, and any element of $V_{Vir}(c,0)$ is a linear combination of terms of the form
\[
L(-m_1) \dots L(-m_r) \mathbf{1},
\]
with $r \in \mathbb{N}$ and $m_1 \geq \dots \geq m_r \geq 2$. The degree of $L(-m_1) \dots L(-m_r) \mathbf{1}$ is defined as $m_1+\dots +m_r$ (cf. \cite[Section 6.1]{Lepowsky-Li} for details on the construction).

The sum $I_{Vir}(c,0)$ of the proper ideals of $V_{Vir}(c,0)$ is again a proper ideal and the quotient
\begin{align*}
L_{Vir}(c,0)= V_{Vir}(c,0)/I_{Vir}(c,0)
\end{align*}
is a simple vertex operator algebra. 

For any integers $p, q$ such that $\on{gcd}(p, q)=1$ and $q > p \geq 2$, we set $c_{p, q}= \displaystyle 1-6\frac{(p-q)^2}{pq}$. It is proved in \cite[Lemma 4.2]{Wang} that if $c \neq c_{p, q}$, then $L_{Vir}(c,0) =V_{Vir}(c,0)$ and thus
\begin{align*}
\begin{array}{ccc}
R(L_{Vir}(c, 0)) &  \stackrel{\cong}{\longrightarrow}  & \cc[x] \\[5pt]
\overline{L(-2)\mathbf{1}} & \longmapsto & x,
\end{array}
\end{align*}
where $\overline{{\color{white} x}}$ indicates the class in the quotient. Also in \cite[Lemma 4.2]{Wang}, it is shown that if $c=c_{p, q}$, then $I_{Vir}(c,0)$ is generated by a single vector $v_{p, q}$ of degree $(p-1)(q-1)$, and the coefficient of the term $L(-2)^{\frac{(p-1)(q-1)}{2}}\mathbf{1}$ in  $v_{p, q}$ is not zero. It follows that $\overline{v_{p, q}}$ is a non-zero multiple of $L(-2)^{\frac{(p-1)(q-1)}{2}}\mathbf{1}  \text{ mod }  C_2(V_{Vir}(c_{p, q},0))$. We thus have an isomorphism 
\begin{align*}
\begin{array}{ccc}
R(L_{Vir}(c_{p, q},0)) &  \stackrel{\cong}{\longrightarrow}  & \cc[x]/ \langle x^{\frac{(p-1)(q-1)}{2}} \rangle \\[5pt]
\overline{L(-2)\mathbf{1}} & \longmapsto & \overline{x}.
\end{array}
\end{align*}

It is proved in \cite{Wang} that the Zhu algebra of the simple Virasoro vertex operator algebra is $A(L_{Vir}(c,0)) \cong \cc[x]$ if $c \neq c_{p, q}$, and that there exists a polynomial $G_{p, q}(x) \in \cc[x]$ of degree $\frac{(p-1)(q-1)}{2}$ such that $A(L_{Vir}(c_{p, q},0)) \cong \cc[x]/\langle G_{p, q}(x) \rangle$. We filter the Zhu algebra by setting $\on{deg} x=2$, and the next result becomes a direct consequence of what we explained thus far.

\begin{proposition} [\cite{Wang}] \label{prop:Wang}
For the simple Virasoro vertex operator algebra $L_{Vir}(c,0)$ with $c \in \cc$, we have
\[
R(L_{Vir}(c,0)) \cong \on{gr} A(L_{Vir}(c,0))  \cong \left\{\begin{array}{l}
\cc[x] \text{ if } c \neq c_{p, q}, \\[10pt]
\cc[x]/ \langle x^{\frac{(p-1)(q-1)}{2}} \rangle \text{ if }  c = c_{p, q}.
\end{array}\right.
\]
\end{proposition}

Based on the above proposition, in order to determine the cohomological variety of $L_{Vir}(c,0)$, we need to compute the Yoneda algebras of the algebra $\cc[x]$ and of the algebra $\cc[x]/\langle x^r \rangle$ for $r \geq 1$.

We can verify that a free resolution of the trivial $\cc[x]$-module $\cc$ is given by the Koszul complex 
\[
0 \longrightarrow \cc[x]  \stackrel{ x}{\longrightarrow} \cc[x]  \longrightarrow \cc \longrightarrow 0.
\]
Furthermore, for any integer $r \geq 2$, a minimal resolution of the trivial $\cc[x]/\langle x^r \rangle$-module $\cc$ is given by
\[
\cdots \stackrel{ x}{\longrightarrow} \cc[x]/\langle x^r \rangle \stackrel{ x^{r-1}}{\longrightarrow} \cc[x]/\langle x^r \rangle  \stackrel{ x}{\longrightarrow} \cc[x]/\langle x^r \rangle  \longrightarrow \cc \longrightarrow 0.
\]
If $r=1$, then $\cc[x]/\langle x^r \rangle \cong \cc$ and so $0 \longrightarrow \cc[x]/\langle x^r \rangle  \stackrel{\on{id}}{\longrightarrow} \cc \longrightarrow 0$ is a free resolution of the $\cc[x]/\langle x^r \rangle$-module $\cc$. We can then compute the Yoneda algebra in those three cases and the next proposition is then well-known.

\begin{proposition}\label{prop:Ext_Vir} 
\begin{enumerate}
\item Set $R=\cc[x]/\langle x^r \rangle$ with $r \geq 1$, and consider the trivial module $\cc \cong R/\langle \overline{x} \rangle$. The Yoneda algebra is 
\[ 
\on{Ext}_{R}^{*}(\cc, \cc) \cong \left\{\begin{array}{l}
\cc  \text{ if }  r=1, \\[5pt]
\cc[\alpha, \beta]/\langle \alpha^2-\delta_{2,r}\beta \rangle \text{ if }  r \geq 2.
\end{array}\right.
\]
\item Set $R=\cc[x]$ and consider the trivial module $\cc \cong R/\langle x \rangle$. The Yoneda algebra is then 
\[ 
\on{Ext}_{R}^{*}(\cc, \cc) \cong \cc[\alpha]/\langle \alpha^2 \rangle.
\]
\end{enumerate}
\delete{
For any $c \in \cc$, the Yoneda algebra of the trivial $R(L_{Vir}(c,0))$-module $\cc$ is given by
\[ 
\on{Ext}_{R(L_{Vir}(c,0))}^{*}(\cc, \cc) \cong \left\{\begin{array}{l}
\cc[\alpha]/\langle \alpha^2 \rangle \text{ if } c \neq c_{p, q}, \\[5pt]
\cc  \text{ if }  c = c_{2, 3}, \\[5pt]
\cc[\alpha, \beta]/\langle \alpha^2-\delta_{2,\frac{(p-1)(q-1)}{2}}\beta \rangle \text{ if }  c = c_{p, q} \text{ with } (p, q) \neq (2, 3).
\end{array}\right.
\]}
\end{proposition}

The Yoneda algebra structures in the proposition are identically the same for $r >2$. It seems that $r$ is not reflected in the Yoneda algebra. However, there is an $A^\infty$-structure which captures $r$ (see \cite{Keller}).

We notice that the Yoneda algebra is always graded commutative. Thus the next corollary is a direct consequence of Definition~\ref{def:cohom_var} and Proposition~\ref{prop:Ext_Vir}.

\begin{corollary}
The cohomological variety of the simple Virasoro vertex operator algebra is given by
\[ 
X_{L_{Vir}(c,0), 0}^{!} \cong \left\{\begin{array}{cl}
\cc & \text{if }  c = c_{p, q} \text{ with } (p, q) \neq (2, 3), \\[5pt]
\{0\} &  \text{otherwise}.
\end{array}\right.
\]
\end{corollary}

\begin{remark}
As mentioned in Remarks~\ref{rem:3.4}, it is straightforward to verify that for any $c \in \cc$, we have $R(V_{Vir}(c,0))=\cc[x]$. As $R(V_{Vir}(c,0))_0=\cc$, it is local and has a unique maximal ideal $\{0\}$. Therefore $H_{\{0\}}^{*}(R(V_{Vir}(c,0)))=\cc[x]/\langle x^2 \rangle$, and so $X_{V_{Vir}(c,0), 0}^{!}=\{0\}$.
\end{remark}

\section{The cohomological varieties of rational affine VOAs}\label{LowerboundforLiealgebras}
We give a brief summary of the construction of the simple affine vertex operator algebras following  \cite{Frenkel-Zhu}, \cite{Kac}, and \cite{Lepowsky-Li}. 

Let $\mathfrak{g}$ be a finite dimensional simple Lie algebra with non-degenerate symmetric invariant bilinear form $\langle \ , \ \rangle_{\mathfrak{g}}$, and let $\mathfrak{h}$ be a Cartan subalgebra of $\mathfrak{g}$. We further assume that $\langle \ , \ \rangle_{\mathfrak{g}}$ is normalised so that the non-degenerate symmetric bilinear form $( \ , \ )_{\mathfrak{h}^*}$ on $\mathfrak{h}^*$ verifies $( \alpha, \alpha )_{\mathfrak{h}^*}=2$ for long roots of $\mathfrak{g}$. The affinization of $\mathfrak{g}$ is
\begin{align*}
\hat{\mathfrak{g}}=\mathfrak{g} \otimes \cc[t, t^{-1}]\oplus \cc \bold{k}
\end{align*}
equipped with the relations
\begin{align*}
[a(m),b(n)]=[a,b](m+n)+m\langle a, b \rangle_{\mathfrak{g}} \delta_{m+n,0}\bold{k}
\end{align*}
for $a,b \in \mathfrak{g}$ and $a(m)=a \otimes t^m$. Furthermore $\bold{k}$ is central in $\hat{\mathfrak{g}}$. We write $\hat{\mathfrak{g}}_{\geq 0}=\mathfrak{g} \otimes t\cc[t] \oplus \mathfrak{g} \oplus \cc \bold{k}$.

Set $k \in \mathbb{C}$ and define $\cc_k$ the $\hat{\mathfrak{g}}_{\geq 0}$-module such that $\mathfrak{g} \otimes t\cc[t] \oplus \mathfrak{g}$ acts as zero and $\bold{k}$ acts as $k$. The universal affine vertex operator algebra of level $k$ is defined as the induced module
\begin{align*}
V_{\hat{\mathfrak{g}}}(k,0)= U(\hat{\mathfrak{g}}) \otimes_{U(\hat{\mathfrak{g}}_{\geq 0})} \cc_k.
\end{align*}

The sum $I_{\hat{\mathfrak{g}}}(k,0)$ of the proper ideals of $V_{\hat{\mathfrak{g}}}(k,0)$ is again a proper ideal and the quotient
\begin{align*}
L_{\hat{\mathfrak{g}}}(k,0)= V_{\hat{\mathfrak{g}}}(k,0)/I_{\hat{\mathfrak{g}}}(k,0)
\end{align*}
is a simple vertex operator algebra. Furthermore, if $k \in \mathbb{N}$, as a $\hat{\mathfrak{g}}$-submodule of $V_{\hat{\mathfrak{g}}}(k,0)$, we have
\begin{align}\label{eq:6.1}
I_{\hat{\mathfrak{g}}}(k,0)= U(\hat{\mathfrak{g}})e_\theta(-1)^{k+1}\mathbf{1},
\end{align}
and
\[ X(m)e_\theta(-1)^{k+1}\mathbf{1}=0
\]
for all $m \geq 1$ and $X \in \mathfrak{g}$, where $e_\theta$ is the root vector of $\mathfrak{g}$ corresponding to the highest root $\theta$.

By definition we have $R(L_{\hat{\mathfrak{g}}}(k,0))=L_{\hat{\mathfrak{g}}}(k,0)/C_2(L_{\hat{\mathfrak{g}}}(k,0))$, which is isomorphic to the quotient $V_{\hat{\mathfrak{g}}}(k,0)/(C_2(V_{\hat{\mathfrak{g}}}(k,0)) + I_{\hat{\mathfrak{g}}}(k,0))\cong R(V_{\hat{\mathfrak{g}}}(k,0))/\overline{I_{\hat{\mathfrak{g}}}(k,0)}$, where $\overline{I_{\hat{\mathfrak{g}}}(k,0)}$ is the image of $I_{\hat{\mathfrak{g}}}(k,0)$ in $R(V_{\hat{\mathfrak{g}}}(k,0))$. It has been discussed in Remarks~\ref{rem:3.4} that for any $k \in \cc$, we have 
\[ R(V_{\hat{\mathfrak{g}}}(k,0)) \cong \on{Sym} (\mathfrak{g}(-1)\overline{\mathbf{1}})\]
with $\mathfrak{g}(-1)=\mathfrak{g} \otimes t^{-1}$.

By identifying $ \mathfrak g (-1)$ with $\mathfrak g$ in degree $1$, we have $R(V_{\hat{\mathfrak{g}}}(k,0)) \cong \on{Sym} (\mathfrak{g})=\cc[\mathfrak g^*]$ as Poisson algebras. Here $\cc[\mathfrak g^*]$ is the coordinate function algebra of the Poisson variety with Poisson structure determined by the Lie bracket $\{x, y\}=[x, y]$ for $x, y \in \mathfrak g$.  Furthermore, the action of $ U(\mathfrak g)$ on $ \overline{e_\theta^{k+1}(-1)\mathbf 1} \in \on{Sym}^{k+1}(\mathfrak g)$ is given by the adjoint $\mathfrak g$-module structure on $\mathfrak g$. We will write $\cdot$ for this action. 

Therefore the $C_2$-algebra of the simple affine vertex operator algebra $ L_{\hat{\mathfrak{g}}}(k,0)$ can be described as
\[
R(L_{\hat{\mathfrak{g}}}(k,0)) \cong \on{Sym} (\mathfrak{g}(-1)\overline{\mathbf{1}})/\overline{I_{\hat{\mathfrak{g}}}(k,0)}.
\]
By \eqref{eq:6.1}, 
\begin{align*}
\overline{I_{\hat{\mathfrak{g}}}(k,0)}=(\on{Sym}(\mathfrak{g}(-1)) \otimes  U(\mathfrak{g})) \overline{e_\theta(-1)^{k+1}\mathbf{1}}=\on{Sym}(\mathfrak g)(U(\mathfrak g)\cdot e_{\theta}^{k+1})\subseteq \on{Sym}(\mathfrak g).
\end{align*}

Note that $\on{Sym}^{k+1}(\mathfrak g)$ is a finite dimensional $\mathfrak g$-module and $ e_\theta^{k+1}$ is a highest weight  vector with highest weight $ (k+1)\theta$. Then $ U(\mathfrak g)\cdot e_\theta^{k+1}=L((k+1)\theta)$ is the irreducible $\mathfrak g$-module.  In particular, since $ U(\mathfrak g)\cdot e_{\theta}^{k+1}$ is a $\mathfrak g$-module via the adjoint action, then $\langle U(\mathfrak g)\cdot e_\theta^{k+1}\rangle=\on{Sym}(\mathfrak g) (U(\mathfrak g)\cdot e_\theta^{k+1})$ is a Poisson ideal of $ \on{Sym}(\mathfrak g)$. Thus we have the following lemma.
\begin{lemma}  \label{lemma:affine_min_gen} For any integer $ k\geq 0 $ and finite dimensional simple Lie algebra $\mathfrak g$,  we have a graded Poisson algebra isomorphism $R(L_{\hat{\mathfrak{g}}}(k,0))\cong \on{Sym}(\mathfrak g)/\langle U(\mathfrak g)\cdot e_\theta^{k+1}\rangle$. 
Any  basis of $ U(\mathfrak g)\cdot e_\theta^{k+1} $ is a minimal set of  generators of the ideal  $\overline{I_{\hat{\mathfrak{g}}}(k,0)}$  of $R(V_{\hat{\mathfrak{g}}}(k,0))$ and they are homogeneous of degree $k+1$.  
\end{lemma}

\delete{By Lemma \ref{l6.1},  $U(\mathfrak{g}) \overline{e_\theta(-1)^{k+1}\mathbf{1}} $ is a highest weight $\mathfrak{g}$-module with highest weight $(k+1)\theta$. Notice that for a positive root $\alpha$ of $\mathfrak{g}$, there exists $N>0$ such that $f_{\alpha}^N e_\theta(-1)^{k+1}\mathbf{1}=0$. Hence $U(\mathfrak{g}) \overline{e_\theta(-1)^{k+1}\mathbf{1}} $ is an integrable $\mathfrak{g}$-module (\cite[\S 10.1]{Kac}). Then, using \cite[Theorem 10.7.b]{Kac} and the fact that $U(\mathfrak{g}) \overline{e_\theta(-1)^{k+1}\mathbf{1}} $ is indecomposable, we see that $U(\mathfrak{g}) \overline{e_\theta(-1)^{k+1}\mathbf{1}} $ is a finite-dimensional irreducible $\mathfrak{g}$-module. 
}
Hence the minimal number of generators of the ideal $\overline{I_{\hat{\mathfrak{g}}}(k,0)}$ inside the polynomial ring $\on{Sym}(\mathfrak{g}(-1))$ is equal to
\begin{equation}\label{eq:Nk-formula}
N_k=\dim_\cc (U(\mathfrak{g}) \overline{e_\theta(-1)^{k+1}\mathbf{1} })=\prod_{\alpha \in \Phi^+}((k+1)\frac{( \theta, \alpha )}{( \rho, \alpha )}+1),
\end{equation}
 by Weyl's character formula. Here $ (\cdot, \cdot)$ is the (symmetric) inner product defining the root system, $ \Phi^+$ is the set of positive roots, and $\rho$ the half-sum of the positive roots. One notes that $ N_k$ is a polynomial of $k$ with nonnegative real coefficients.  In particular, if $\mathfrak{g}=\on{sl}_2({\mathbb C})$, then  $\theta=2\rho$ and $N_k=2k+3$. When the integer level satisfies $k \geq 2$, we can apply Theorem~\ref{SliceofMaxH(R)ab} and obtain a lower bound on the dimension of $X_{L_{\hat{\mathfrak{g}}}(k,0),0}^!$. 
 
\begin{theorem}\label{dimX_V!LieAlgebra}
Let $\mathfrak{g}$ be a finite dimensional simple Lie algebra. If $k$ is a positive integer larger or equal to $2$, then there exists closed embedding of algebraic varieties. 
\begin{align*}
\cc^{N_k} \myhookrightarrow X_{L_{\hat{\mathfrak{g}}}(k,0), 0}^{!}.
\end{align*}

\end{theorem}

\begin{remark}
It was briefly mentioned in Remarks~\ref{rem:3.4}, but one can verify that for any $k \in \cc$, we have $R(V_{\hat{\mathfrak{g}}}(k,0))=\on{Sym}(\mathfrak{g})$. As $\on{Max}(R(V_{\hat{\mathfrak{g}}}(k,0))_0)=\{0\}$, we get $H_{\{0\}}^{*}(R(V_{\hat{\mathfrak{g}}}(k,0)))=
\wedge(\mathfrak{g})$, and therefore $X_{V_{\hat{\mathfrak{g}}}(k,0), 0}^{!}=\{0\}$.
\end{remark}

If the number $N_k$ in the above theorem is strictly bigger than the dimension of $\mathfrak{g}$, then this presentation of the algebra $R(L_{\hat{\mathfrak{g}}}(k,0))$ has more relations than variables, and so the ideal of relations cannot be generated by a regular sequence. We will use this reasoning to show the following proposition:

\begin{proposition}\label{R(V)completeintersection} 
Let $\mathfrak{g}$ be a finite dimensional simple Lie algebra and set $k \in \mathbb{N}$. Then we have the following equivalence:
\begin{align*}
\widetilde{X}_{L_{\hat{\mathfrak{g}}}(k,0)} \subset \mathfrak{g}^* \text{ is a complete intersection if and only if } k=0.
\end{align*}
\end{proposition} 

\begin{proof}
\delete{For any $k \geq 0$,  $U(\mathfrak{g}) \overline{e_\theta(-1)^{k+1}\mathbf{1}}$ is the irreducible highest weight $\mathfrak{g}$-module $L((k+1)\theta)$ of highest weight $(k+1)\theta$, which is dominant. In particular, for $k=0$, we have $L(\theta)= U(\mathfrak{g}) \overline{e_\theta(-1)\mathbf{1}}=\mathfrak{g}(-1)\overline{\mathbf{1}} \cong \mathfrak{g}$ as adjoint module.} 
\delete{
Using Weyl's character formula, we obtain
\begin{equation} \label{eq:Nk-formula}
N_k=\prod_{\alpha \in \Phi^+}\frac{( (k+1)\theta+\rho, \alpha)}{\langle \rho, \alpha \rangle}=\prod_{\alpha \in \Phi^+}((k+1)\frac{( \theta, \alpha )}{( \rho, \alpha )}+1)
\end{equation}
with $\Phi^+$ the positive roots of $\mathfrak{g}$, and }

We note that  $N_0=\on{dim}_\cc \mathfrak{g}$ since $ \mathfrak g=L(\theta)$ as $ \mathfrak g$-module. Furthermore, since $(\theta, \theta ) >0$ and $\theta$ is a positive root, we know that $N_k$ is a non constant polynomial in $k$ with positive real coefficients. Thus $N_k$ is a strictly increasing function of $k$ on the integer values $k \geq 0$. Therefore when $k \geq 1$, we have $N_k > N_0=\on{dim}_\cc \mathfrak{g}$. 

By Lemma \ref{lemma:affine_min_gen}, the ideal $\langle U(\mathfrak g)\cdot e_\theta^{k+1}\rangle$ is generated by $N_k$ $\cc$-linearly independent elements of the same degree $ k+1$. Let $ p_{k+1}: \on{Sym}(\mathfrak g)\to \on{Sym}^{k+1}(\mathfrak g)$ be the linear projection map. Then, for any generating set $S$ of $\overline{I_{\hat{\mathfrak{g}}}(k,0)}$, $p_{k+1}(S)$  must  generate a $ N_k$-dimensional subspace in $ \on{Sym}^{k+1}(\mathfrak g)$.  Therefore $p_{k+1}(S)$  have  at least  $ N_k> \dim \mathfrak g$ many non-zero elements. Thus $S$ cannot be a regular sequence. Hence the inclusion $\widetilde{X}_{L_{\hat{\mathfrak{g}}}(k,0)} \subset \mathfrak{g}^*$ is not a complete intersection when $k \geq 1$. 

 \delete{Thus the number of minimal generators of the ideal of relations is greater than the number of variables, and so the ideal cannot be generated by a regular sequence.}
 
If $k=0$, then $U(\mathfrak{g}) \overline{e_\theta(-1)\mathbf{1}}=\mathfrak{g}(-1)\overline{\mathbf{1}}$, and thus $R(L_{\hat{\mathfrak{g}}}(0,0)) \cong \on{Sym}(\mathfrak{g}(-1)\overline{\mathbf{1}})/ \langle \mathfrak{g}(-1)\overline{\mathbf{1}} \rangle$. Therefore the inclusion $\widetilde{X}_{L_{\hat{\mathfrak{g}}}(0,0)} \subset \mathfrak{g}^*$ is a complete intersection.
\end{proof}

\begin{remark} By \eqref{eq:Nk-formula}, $ N_k$ is a polynomial of degree $|\{ \alpha\in \Phi^+\;|\; (\theta, \alpha)>0\}|$. Here $(\theta, \alpha)$ is the standard symmetric inner product. In the following, we give a list of the degrees for the simple root systems (we use Bourbaki's notation \cite{Bourbaki}): the linear factor corresponding to a positive root $\alpha $ has coefficient $\frac{(\theta, \alpha)}{(\rho, \alpha)}$. Note that writing $ \alpha=\sum_{i}n_i\alpha_i$ with respect to simple roots $\{\alpha_1, \dots, \alpha_l\}$, we have 
\[(\rho, \alpha) =\sum_{i}n_i(\rho, \alpha_i)=\sum_{i}\frac{n_i (\alpha_i, \alpha_i)}{2}\]
by noting $ (\rho, \alpha_i^\vee)=1$ and $\alpha^\vee=2\alpha/(\alpha, \alpha)$.
Similarly,  
\[(\theta, \alpha)=\sum_i n_i (\theta, \alpha_i).
 \]
 Thus one only needs the information $(\alpha_i, \alpha_j)$. We give below the degrees of $N_k$ as a polynomial in $k$ for simple Lie algebras:
 
 \begin{center}
 $ \renewcommand\arraystretch{1.2}\begin{array}{|c|c|}
 \hline
 \text{Type of } \mathfrak g & \text{Degree of }N_k \\
  \hline
 A_l  & 2l-1 \\
  \hline
B_l &  4l-5\\
 \hline
C_l & 2l-1 \\
 \hline
D_l & 4l-7 \\
 \hline
E_6 & 21 \\
 \hline
E_7 & 33 \\
 \hline
E_8 & 57 \\
 \hline
F_4 & 15 \\
 \hline
G_2 &5 \\
\hline
 \end{array}$
 \end{center}

 \delete{
 $A_l$:  $\theta=\varepsilon_1-\varepsilon_{l+1}$ and $\alpha=\varepsilon_i-\varepsilon_j$ with $ i<j$ such that $ (\theta, \alpha)>0$ if and only if $ i=1 $ or $ j=l+1$. Hence, $\deg_k(N_k)=(l(l+1)-(l-1)(l-2))/2=2l-1$. 
 
 $B_l$ ($l\geq 2$): $\theta=\varepsilon_1+\varepsilon_2$;  the positive roots  $\varepsilon_i $ with  $1\leq i\leq l$ and $\varepsilon_i\pm \varepsilon_j $ with $1\leq i<j\leq l$. Hence, $\deg_k(N_k)=3(l-2)+l-1+2=4l-5$. 
 
$C_l$ ($l\geq 2$): $\theta=2\varepsilon_1$. the  positive roots are $ 2\varepsilon_i $ with  $1\leq i\leq l$ and $\varepsilon_i\pm \varepsilon_j $ with $1\leq i<j\leq l$. Hence, $\deg_k(N_k)=2(l-1)+1=2l-1$. 
  
$D_l$ ($l\geq 3$):   $\theta=\varepsilon_1+\varepsilon_2$;  the positive roots $\varepsilon_i\pm \varepsilon_j $ with $1\leq i<j\leq l$. Hence, $\deg_k(N_k)= 4l-7$. 
  
$E_6$:  $\theta=\omega_2$ (the fundamental weight). Thus $ \deg_k(N_k)$ is the number of positive roots with simple root $\alpha_2$ as a component, i.e. $\deg_k(N_k)=21$.
  
$E_7$: $\theta=\omega_1$, so $\deg_k(N_k)=33$. In Bourbaki's list, there are 18 positive roots with a coefficient at least $2$ and with $\alpha_7$,  including 3 where the coefficient of $\alpha_1$ is 0. There are 8 from $E_6$ which give those without $\alpha_7$ but with $\alpha_1$. Furthermore, there are 10 positive roots with $\alpha_1$ and coefficients at most $1$. 

$E_8$: $\theta=\omega_8$, so $\deg_k(N_k)=57$. There are 47 in the list with a coefficient at least $2$ and they all contain $\alpha_8$. There are 10 positive roots with all coefficients at most 1 and with $\alpha_8$. Note that positive roots of $E_7$ have no impact since they do not contain $\alpha_8$. 

$F_4$: $\theta=\omega_1=\varepsilon_1+\varepsilon_2$, so $\deg_k(N_k)= 15$.

$G_2$: $\theta=\omega_2$, so  $\deg_k(N_k)= 5$.}

Here are the rank $2$ cases: 

$G_2$: $N_k=(k+2)(\frac{3}{4}(k+1)+1)(\frac{3}{5}(k+1)+1)(\frac{1}{2}(k+1)+1)(\frac{2}{3}(k+1)+1)$. 

\delete{{\color{red} I think it should be $N_k=(k+2)(\frac{1}{2}(k+1)+1)(\frac{1}{3}(k+1)+1)(\frac{1}{4}(k+1)+1)(\frac{2}{5}(k+1)+1)$. }}

$B_2$: $N_k=(2(k+1)+1)(k+2)(\frac{2}{3}(k+1)+1)$.

\delete{{\color{red} I think it should be $N_k=(2(k+1)+1)(k+2)(\frac{4}{3}(k+1)+1))$. }}

$A_2$:  $N_k=(k+2)^3$.
\end{remark}
\begin{remarks}
(1) We saw in Proposition~\ref{prop:Wang} that for the rational vertex operator algebra $L_{Vir}(c_{p, q},0)$, the $C_2$-algebra is a complete intersection ring. Furthermore, when $k \in \mathbb{Z}_+$, Proposition~\ref{R(V)completeintersection} states that the $C_2$-algebra of the rational affine vertex operator algebra $L_{\hat{\mathfrak{g}}}(k,0)$ is not a complete intersection. It is however a Cohen-Macaulay ring, because it is finite dimensional. There exists an intermediate class of rings between complete intersection and Cohen-Macaulay, namely Gorenstein rings. They are characterised through the injective dimensions of their localisations at prime ideals. An interesting question would be to determine if $R(L_{\hat{\mathfrak{g}}}(k,0))$ is a Gorenstein ring for $k \in \mathbb{Z}_+$. 

(2) Another question we have not addressed in this paper is the conditions for the Ext algebra $\on{Ext}_{R(V)}^*(\bb C, \bb C)$ to be finitely generated. For finite group schemes, this is a result of Friedlander-Suslin (\cite{Friedlander-Suslin}). This is important to ensure that the variety $X_{V}^!$ is of finite type. Clearly, when $ R(V)$ is locally complete intersection at the vertex point, then the Tate resolution we constructed in Section~\ref{sec:4.4} would be $X^2$ and thus the Yoneda algebra is finitely generated. But this is not necessary. As one can see from the construction of the Tate resolution, if $X=X^{n}$ for some $ n$ (we will say that such an ideal $I$ (or the closed subscheme $\on{Spec}(R/I)$ in $ \on{Spec}(R)$) is $n$-regular), then the Yoneda algebra is finitely generated. These questions will be discussed in a forthcoming paper.

(3)  We have only considered the simple affine vertex operator algebras $ L_{\hat{\mathfrak g}}(k,0)$ for positive integer levels.  For other levels $k$, such as  admissible levels, the $C_2$-algebras $R(L_{\hat{\mathfrak g}}(k, 0))$ and the associated varieties have been discussed (\cite{Arakawa2, Arakawa-Moreau}). It would be interesting to find approximations of the cohomological varieties. 

(4) There are other rational vertex operator algebras arising from affine rational vertex operator algebras through coset constructions such as parafermion (\cite{Arakawa-Lam-Yamada}), $\cal W$-algebras, and Monster moonshine modules (\cite{Frenkel-Lepowsky-Meurman}). These algebras are not generated by degree 1 elements, and we still do not know anything about their cohomological varieties. 
\end{remarks}

\begin{remark} It was pointed out by David Ridout shortly after the paper was posted on  ArXiv that $\deg(N_k)=2h^\vee-3$. Here $ h^\vee$ is the dual Coxeter number. This can be easily checked from the table in \cite[p.80]{Kac}. In fact, this follows from \cite[Prop.1]{Suter}.
\end{remark}

\setlength\bibitemsep{7pt}

\end{document}